%% file: UniformKummerBound.tex
\documentclass[11pt,a4paper]{amsart}
\usepackage[utf8]{inputenc}
\usepackage[english]{babel}
\usepackage{amsmath}
\usepackage{amsfonts}
\usepackage{amssymb}
\usepackage{amsthm}
\usepackage[dvipsnames]{xcolor}
\usepackage[normalem]{ulem}
\usepackage{pdflscape}

\usepackage[top=3.5cm,bottom=3.5cm,left=3cm,right=3cm]{geometry}

\usepackage{hyperref}
\usepackage{tikz-cd}
\usepackage{enumitem}

\usepackage{tikz-cd}

\newtheorem {theorem}{Theorem}
\newtheorem {corollary}[theorem]{Corollary}
\newtheorem {lemma}[theorem]{Lemma}
\newtheorem {proposition}[theorem]{Proposition}
\theoremstyle{definition}
\newtheorem {remark}[theorem]{Remark}
\newtheorem {definition}[theorem]{Definition}

\numberwithin{theorem}{section}

\DeclareMathOperator{\Aut}{Aut}

\DeclareMathOperator{\Gal}{Gal}
\DeclareMathOperator{\Hom}{Hom}

\DeclareMathOperator{\Mat}{Mat}
\DeclareMathOperator{\GL}{GL}
\DeclareMathOperator{\SL}{SL}
\DeclareMathOperator{\PSL}{PSL}
\DeclareMathOperator{\PGL}{PGL}

\DeclareMathOperator{\lcm}{lcm}
\DeclareMathOperator{\Id}{Id}

\DeclareMathOperator{\tors}{tors}

\DeclareMathOperator{\ab}{ab}

\DeclareMathOperator{\tr}{tr}

\newcommand{\Q}{\mathbb{Q}}

\newcommand{\Z}{\mathbb{Z}}
\newcommand{\F}{\mathbb{F}}

\newcommand{\matr}[4]{\left(\begin{array}{cc}#1 & #2 \\ #3 & #4\end{array}\right)}

\newcommand{\set}[1]{\left\{ #1 \right\}}

\renewcommand{\geq}{\geqslant}
\renewcommand{\leq}{\leqslant}

\begin{document}

\title{Some uniform bounds for elliptic curves over $\mathbb{Q}$}

\author{Davide Lombardo}
\address[]{Dipartimento di Matematica, Universit\`a di Pisa,
Largo Bruno Pontecorvo 5, 56127 Pisa, Italy}
\email{davide.lombardo@unipi.it}

\author{Sebastiano Tronto}
\address[]{Department of Mathematics, University of Luxembourg,
6 av.\@ de la Fonte, 4364 Esch-sur-Alzette, Luxembourg}
\email{sebastiano.tronto@uni.lu}

\begin{abstract}
We give explicit uniform bounds for several quantities relevant to the study of Galois representations attached to elliptic curves $E/\Q$. We consider in particular the subgroup of scalars in the image of Galois, the first Galois cohomology group with values in the torsion of $E$, and the Kummer extensions generated by points of infinite order in $E(\Q)$.
\end{abstract}

\maketitle

\input{1-IntroAndPrelims}
\input{3-ScalarsInImage}

\input{4-ExponentH1}

\input{5-AlgebraA}

\input{6-KummerBound}
\input{8-Examples}

\appendix

\input{7-3adic}

\bibliographystyle{acm}
\bibliography{biblio}

\end{document}

%% file: 1-IntroAndPrelims.tex
\section{Introduction}
Let $E/\Q$ be an elliptic curve. Our purpose in this paper is to provide universal bounds on several arithmetically relevant quantities attached to $E$, and more precisely to its Galois representations. For each prime $\ell$ we denote by $G_{\ell^\infty}$ the image of the $\ell$-adic Galois representation attached to $E/\Q$, and by $G_\infty$ the image of the adelic representation (see Section \ref{subsect:NotationTorsion} for details). We provide in particular:
\begin{enumerate}
\item a uniform upper bound for the index $[\Z_\ell^\times : \Z_\ell^\times \cap G_{\ell^\infty}]$ (Theorem \ref{thm:ScalarsImageGalois}), that is, we show that for every prime $\ell$ the subgroup of scalars in the $\ell$-adic image of Galois contains a fixed subgroup of $\Z_\ell^\times$ for all elliptic curves $E/\Q$;
\item a uniform upper bound on the exponent of the cohomology groups $H^1(G_\infty, E[N])$, for all positive integers $N$ (Theorem \ref{theorem:boundCohomologyOverQ});
\item a uniform lower bound for the closed $\Z_\ell$-subalgebra $\Z_\ell[G_{\ell^\infty}]$ of $\Mat_{2 \times 2}(\Z_\ell)$ generated by $G_{\ell^\infty} \subseteq \GL_2(\Z_\ell) \subset \Mat_{2 \times 2}(\Z_\ell)$: for each prime $\ell$ we compute an optimal exponent $m_\ell$ such that $\Z_\ell[G_{\ell^\infty}]$ contains $\ell^{m_\ell} \Mat_{2 \times 2}(\Z_\ell)$ (Theorem \ref{theorem:algebraZG});
\item a uniform lower bound on the degrees of the relative `Kummer extensions' (Section \ref{sect:KummerDegrees}), that is, the extensions $\Q( \frac{1}{N}\alpha, E[N] ) / \Q( E[N] )$ obtained by adjoining all $N$-torsion points of $E$ and all $N$-division points of a fixed rational point $\alpha \in E(\Q)$ (Theorem \ref{theorem:mainTheoremKummer}), provided that $\alpha$ and all its translates by torsion points are not divisible by any $d>1$ in the group $E(\Q)$.
\end{enumerate}

We now elaborate on each of these four topics. It is well-known that, for a fixed prime $\ell$ and number field $K$, the images of the $\ell$-adic Galois representations attached to non-CM elliptic curves over $K$ admit a uniform upper bound for the index $[\GL_2(\Z_\ell) : G_{\ell^\infty}]$ (see for example \cite{MR2434156}). Since the CM case is easy to handle, this implies the existence of a bound as in (1). However, the result of \cite{MR2434156} is not effective, and a great deal of work has gone into classifying the possible $\ell$-adic images of Galois even just for elliptic curves over $\Q$ (the so-called `Program B' of Mazur), see for example \cite{MR488287, MR3500996, PossibleImages, MR2753610, PedroSamuel, Isogeny7, Isogeny}. Our results on (1), which rely heavily on many of these previous developments, give a complete answer for all primes $\ell \neq 3$, and a rather sharp bound also for the remaining case $\ell=3$. With the exception of the case $\ell=2$, that was already treated in \cite{MR3500996}, we prove our estimates by group-theoretic means (see in particular the criteria given by Corollary \ref{cor:AllScalars} and Proposition \ref{prop:padicscalars}). The advantage of such an approach is that our methods can easily be extended to number fields other than $\Q$. The price to pay is that we don't get the sharpest possible result for $\ell=3$, a direction we have decided not to pursue further also due to upcoming work of Rouse, Sutherland and Zureick-Brown on the complete classification of $3$-adic images of Galois for elliptic curves over $\Q$ with a rational $3$-isogeny (see also Remark \ref{rem:RouseExponent3}).

Concerning (2), there is already a significant past literature on controlling the cohomology groups $H^1(G_{\ell^\infty}, E[\ell^k])$, see for example \cite{LawsonWutrich}, \cite[Lemma 10]{MR286754} and \cite[Section 3]{MR1628193}. Kolyvagin's celebrated work on the Birch--Swinnerton-Dyer conjecture also needs to rely on vanishing statements for the Galois $H^1$ of the $\ell$-torsion of elliptic curves \cite[Proposition 9.1]{MR1110395}. In this paper we go beyond the known results in two different ways. On the one hand, we extend the statements in \cite{LawsonWutrich} by giving a uniform upper bound on the exponents of all the cohomology groups $H^1(G_{\ell^\infty}, E[\ell^k])$, where \cite{LawsonWutrich} mostly gave vanishing conditions and did not extensively treat the cases when the cohomology does not vanish. As we show in Section \ref{sect:Examples}, these results for a fixed prime $\ell$ are rather sharp. Secondly, and more importantly for our application (4), we also treat the Galois action on the $N$-torsion of elliptic curves when $N$ is not necessarily a prime power. While the case $N=\ell^k$ follows easily from the existence of non-trivial scalars in the image of Galois, the general case introduces a number of additional complications, connected with the possible `entanglement' of torsion fields at different primes. Since not even the classification of possible $\ell$-adic images is complete, the problem of describing all possible entanglements between torsion fields seems to be out of reach for the moment (but see \cite{MR3957898}, \cite[\S 3]{campagna2019cyclic}, \cite{campagna2020entanglement} and \cite{daniels2021classification} for some positive results), 
so the computation of $H^1(G_\infty, E[N])$ cannot be approached directly.
We are still able to obtain useful information on this group (in particular, prove Theorem \ref{theorem:boundCohomologyOverQ}) by using the inflation-restriction exact sequence and controlling the amount of entanglement by using our results on scalars and the uniform bound on the degrees of prime-degree isogenies (Mazur's theorem). As in the case of (1), the intermediate technical results on the way to the proof of Theorem \ref{theorem:boundCohomologyOverQ} should hopefully apply in more general situations (see in particular Proposition \ref{proposition:cohomologyBoundGeneric}). Our numerical estimate on the exponent of $H^1(G_\infty, E[N])$ is nowhere near as sharp as the corresponding bounds for the special case $N=\ell^k$, but notice that (unlike that case) it is not a priori clear that a uniform bound should even exist. We had in fact already shown the existence of such a bound in \cite{2019arXiv190905376L}, but the result was not effective.

We remark that we have chosen to formulate our bounds in terms of divisibility: we prove that multiplication by a suitable universal constant $e$ kills the abelian group $H^1(G_\infty, E[N])$, and therefore the exponent of this group divides $e$. The numerical constant would be much smaller if we instead formulated the result as an inequality (that is, if we were content with knowing that the exponent of $H^1(G_\infty, E[N])$ does not exceed a certain constant $e'$), but we feel that our version will be more useful in applications.
In particular, we would like to stress that -- even ignoring the non-effective parts of the argument -- the ideas of \cite{2019arXiv190905376L} would lead to a (divisibility) bound for $H^1(G_\infty, E[N])$ involving primes up to several millions, while the value of $e$ that we find with the new, more streamlined proof given in the present paper is only divisible by the primes up to $11$ (which, as we show in Section \ref{sect:Examples}, all need to appear as factors of $e$). In other words, while our constant $e$ is probably not optimal, it is at least supported on the correct set of primes.

The algebra $\Z_\ell[G_{\ell^\infty}]$ considered in (3) is also a classical object in the field of Galois representations, and its analogues in arbitrary dimension most famously play an important role in Faltings's proof of his finiteness theorems for abelian varieties. While in many applications one needs control over the actual image of Galois $G_{\ell^\infty}$, in several cases it is enough to get a handle on the sub-algebra of $\Mat_{2 \times 2}(\Z_\ell)$ generated by it. In the hope that it will be useful in such cases, we give explicit values $m_\ell$ with the property that $\ell^{m_\ell} \Mat_{2 \times 2}(\Z_\ell)$ is contained in $\Z_\ell[G_{\ell^\infty}]$ for all elliptic curves $E/\Q$, and we show that these values are optimal.

Finally, (4) was our original motivation for the work done in this paper: we had already shown a similar result in \cite{2019arXiv190905376L}, but (lacking all the previous information (1), (2), (3)) we could not make it explicit, or in fact even effective. With all the preliminary work done in \cite{2019arXiv190905376L} and in the other sections of this paper, the desired result on Kummer extensions is now easy to prove.
Notice that the assumption on the (in)divisibility of the point $\alpha$ is necessary: if $\alpha=N\beta$ for some rational point $\beta$ then $\Q(\frac1N\alpha, E[N])$ coincides with the torsion field $\Q(E[N])$, and clearly no non-trivial lower bound for $[ \mathbb{Q}(\frac{1}{N}\alpha, E[N]) : \mathbb{Q}(E[N]) ]$ exists in this case. On the other hand, it is possible to relax this assumption if one is willing to accept a bound that depends on the largest integer $d$ such that $\alpha$ is $d$-divisible in $E(\Q)/E(\Q)_{\tors}$, but not on the curve $E$, see \cite[Remark 7.2]{2019arXiv190905376L}.

We make two final comments. In order to get completely uniform results, we also need to treat the case of CM elliptic curves: while the proofs are generally easier than their non-CM counterparts, they are genuinely different and require some additional observations. In several cases we also prove sharper results in this context (see in particular Theorem \ref{thm:CMH1Bound} for a bound on the cohomology groups attached to CM elliptic curves over number fields). For this reason, while it is clear that one can obtain uniform statements that do not distinguish between CM and non-CM curves (essentially, by taking the maximum of the bounds in the two cases), we have chosen to formulate most of our results with a clear distinction between the two situations.

Finally, we would like to point out that much of what we do in this paper can be extended to number fields $K$ having at least one real place, at least if one is ready to believe the Generalised Riemann Hypothesis. 
Indeed, under GRH, the uniform boundedness of isogenies of elliptic curves over $K$ holds by \cite[Corollary 6.5]{MR3211798}. Concerning the four topics above, we have already pointed out that (1) is known to be true for all number fields, and the group-theoretic criteria of Propositions \ref{proposition:containsAllScalars} and \ref{prop:padicscalars} can in most cases make this explicit (in terms of a bound on the possible degrees of cyclic isogenies). As for (2), the proof of Theorem \ref{theorem:boundCohomologyOverQ} can be repeated almost verbatim once one knows that the subgroup of scalars in $G_{\ell^\infty}$ is uniformly lower-bounded for all $\ell$ and that the degrees of cyclic isogenies are also bounded. A bound as in (3) follows from Proposition \ref{proposition:algebraReducibleGeneral}, Proposition \ref{prop:Algebra2General} and Corollary \ref{corollary:algebraIrreducible}. Finally, by the results of \cite{2019arXiv190905376L} a bound as in (4) can be obtained as a consequence of all the above. We do not pursue this observation further since the result would in any case be conditional on GRH, but we hope to have convinced the reader that the techniques in this paper have wider applicability than just the case of rational numbers.

\subsection{Structure of the paper}
In Section \ref{sect:Preliminaries} we recall some basic properties of $\ell$-adic numbers and of subgroups of $\GL_2(\F_\ell)$ for $\ell$ a prime number. We also introduce our notation for the Galois representations attached to elliptic curves. In Section \ref{section:scalarsInImage} we prove our first main results, Theorems \ref{thm:ScalarsImageGalois} and Proposition \ref{prop:ScalarsInImageCM}, which give a uniform lower bound for the subgroup of scalars in the image of Galois representations attached to elliptic curves over $\mathbb{Q}$ (in the non-CM and CM case respectively). In Section \ref{sect:CohomologyBound} we deduce from this an estimate on the exponent of the first cohomology group for the action of Galois on the torsion points of an elliptic curve $E/\Q$, see Theorem \ref{theorem:boundCohomologyOverQ} and Theorem \ref{thm:CMH1Bound} (which covers the CM case for elliptic curves over arbitrary number fields). In Section \ref{sect:GaloisImageAlgebra} we describe the $\Z_\ell$-subalgebra of $\operatorname{End}(\Z_\ell^2)$ generated by the image of an $\ell$-adic Galois representation attached to an elliptic curve over $\Q$. Finally, in Section \ref{sect:KummerDegrees} we combine the previous results to study the Kummer theory of elliptic curves over $\Q$, leading to a uniform estimate on the degrees of Kummer extensions (Theorem \ref{theorem:mainTheoremKummer}). Section \ref{sect:Examples} gives some explicit examples showing that at least some of our estimates are not too far from optimal. The group-theoretic Appendix \ref{sect:pAdicAppendix} contains the proof of an auxiliary result needed in Section \ref{section:scalarsInImage} to study the case of $3$-adic Galois representations.

\subsection{Acknowledgements}
We thank Peter Bruin for providing us with a reference for Lemma \ref{lemma:AbsIrred}, and Andrea Maffei for a useful discussion on reductive groups.
We also thank Jeremy Rouse and Michael Cerchia for fruitful discussions, for informing us of their work in progress, and
for suggesting some improvements to our results.

\section{Preliminaries}\label{sect:Preliminaries}

\subsection{\texorpdfstring{$\ell$}{l}-adic numbers}

For every prime $\ell$ we denote by $\Z_\ell$ the ring of $\ell$-adic integers,
which we regard as a profinite (topological) ring, and by $v_\ell$ the
$\ell$-adic valuation on $\Z_\ell$. We denote by $\Z_\ell^+$ the underlying
abelian group of $\Z_\ell$, which is topologically generated by any element
of $\ell$-adic valuation $0$, and by $\Z_\ell^\times$ its group of units.
For $n\geq 1$ we let $1+\ell^n\Z_\ell=\set{x\in\Z_\ell\mid v_\ell(x)\geq n}$. Since the subgroup $\ell^n \Z_\ell$ of $\Z_\ell^+$ is topologically generated by any element of valuation $n$, from \cite[Proposition 4.3.12]{MR2312337} one obtains:
\begin{lemma}
  \label{lemma:ScalarsValuationOne}
  Let $n$ be a positive integer and let $\ell>2$ be a prime. Let $G$ be a
  closed subgroup of $\Z_\ell^\times$. If there is $\lambda\in G$ such that
  $v_\ell(\lambda-1)=n$, then $G$ contains $1+\ell^n\Z_\ell$.
\end{lemma}

There is group homomorphism $\F_\ell^\times\to\Z_\ell^\times$, the \emph{Teichm\"uller lift}, that sends every $\lambda\in \F_\ell^\times$ to the unique
$\tilde\lambda\in \Z_\ell^\times$ such that $\tilde\lambda^\ell=\tilde\lambda$
and $\tilde\lambda\equiv\lambda\pmod\ell$ (such a $\tilde{\lambda}$ exists by Hensel's
lemma). 
The following well-known lemma (see e.g.~\cite[Corollary 4.5.10]{gouvea1997p}) shows that $\Z_\ell^\times$ is generated
by $1+\ell \Z_\ell$ and by the Teichm\"uller lifts of
all elements of $\F_\ell^\times$, a fact that will be used in Section
\ref{section:scalarsInImage}.
\begin{lemma} \label{lemma:TeichSplit}
  The short exact sequence
  \begin{align*}
    1\to 1+\ell\Z_\ell\to\Z_\ell^\times\to\F_\ell^\times\to 1
  \end{align*}
  is split by the Teichm\"uller lift.
\end{lemma}

If $m$ and $n$ are positive integers we
extend $v_\ell$ to the additive group of $m\times n$ matrices with coefficients in $\Z_\ell$ as follows: if
$A=(a_{ij})_{1\leq i\leq m,\,1\leq j\leq n}\in\Mat_{m\times n}(\Z_\ell)$
we let $v_\ell(A):=\min\set{v_\ell(a_{ij})\mid1\leq i\leq m,\,1\leq j\leq n}$.
In particular, for $m=n$ we obtain a valuation $v_\ell$ on the ring
$\Mat_{n\times n}(\Z_\ell)$.
The following is proven by an immediate induction on $v_\ell(n)$:
\begin{lemma}
  \label{lemma:ellToTheN}
  Let $s$ be a positive integer and let $h\in \GL_s(\Z_\ell)$.
  If $v_\ell(h-\Id)>0$, then $v_\ell(h^n-\Id)>v_\ell(n)$ for all positive integers $n$.
\end{lemma}

\subsection{Cartan subgroups\texorpdfstring{ of $\GL_2(\F_\ell)$}{}}\label{subsect:CartanSubgroups}

We recall the definition and basic properties of Cartan subgroups of $\GL_2(\F_\ell)$ when $\ell$ is an odd prime.
\begin{definition}
  Let $\ell> 2$ be a prime and let $\delta\in\F_\ell^\times$. We call
  \begin{align*}
    C_\ell(\delta):=\set{\matr x{\delta y}yx\mid x,y\in\F_\ell,\,
                        x^2-\delta y^2\neq0}\subseteq \GL_2(\F_\ell)
  \end{align*}
  the \emph{Cartan subgroup of $\GL_2(\F_\ell)$ with parameter $\delta$}.
  We call $C_\ell(\delta)$ \emph{split} if $\delta$ is a square in $\F_\ell$, and \emph{nonsplit} otherwise.
  We also denote by $N_\ell(\delta)$ the normalizer of $C_\ell(\delta)$
  in $\GL_2(\F_\ell)$.
\end{definition}

\begin{remark}
Let $\lambda \in \mathbb{F}_\ell^\times$. Conjugating $C_\ell(\delta)$ by $\begin{pmatrix}
\lambda & 0 \\
0 & 1
\end{pmatrix}$ gives $C_\ell(\delta \lambda^2)$, so that a Cartan subgroup is determined (up to conjugacy in $\GL_2(\mathbb{F}_\ell)$) by the class of $\delta \in \mathbb{F}_\ell^\times / \mathbb{F}_\ell^{\times 2}$, that is, only by whether or not $\delta$ is a square in $\mathbb{F}_\ell^\times$.
\end{remark}

\begin{lemma}[{\cite[Lemma 14]{MR3690236}}]
  Let $\ell> 2$ be a prime and let $\delta\in\F_\ell^\times$.
  The Cartan subgroup $C_\ell(\delta)$ has index $2$ in $N_\ell(\delta)$.
  More precisely, we have
  \begin{align*}
    N_\ell(\delta)=C_\ell(\delta)\cup \matr100{-1}\cdot C_\ell(\delta)\,.
  \end{align*}
\end{lemma}

\begin{remark} \label{remark:newCartanModel}
  Let $\ell> 2$ be a prime and let $\delta\in \F_\ell^\times$. Considering
  the matrix $g=\matr111{-1}$, whose inverse is $\frac12g$, one sees that
  $C_\ell(1)$ is conjugated to the subgroup
  \begin{align*}
    C_\ell^*(1):=gC_\ell(1)g^{-1}=\set{\matr t00w\mid t,w\in\F_\ell^\times}
  \end{align*}
  of $\GL_2(\F_\ell)$, whereas for $\delta\neq 1$ it is conjugated to
  \begin{align*}
    C_\ell^*(\varepsilon):=gC_\ell(\delta)g^{-1}=
    \set{\matr{x+\varepsilon w}{-w}{w}{x-\varepsilon w}\mid x,w\in\F_\ell,\,
    x^2+(1-\varepsilon^2)w^2\neq0}
  \end{align*}
  where $\varepsilon=\frac{\delta+1}{\delta-1}$.
  Similarly, $N_\ell(\delta)$ is conjugated to
  \begin{align*}
    N^*_\ell(\varepsilon)=C^*_\ell(\varepsilon)\cup \matr0110\cdot
    C^*_\ell(\varepsilon)\,,
  \end{align*}
  which is the normalizer of $C_\ell^*(\varepsilon)$.
\end{remark}

\subsection{Subgroups of \texorpdfstring{$\GL_2(\F_\ell)$ and $\GL_2(\Z_\ell)$}{the $2$-dimensional linear group}}\label{subsect:GroupTheory}
Since we will need to rely on it several times throughout the paper, we remind the reader of the well-known classification of maximal subgroups of $\GL_2(\F_\ell)$, traditionally attributed to Dickson. For $\ell=2$ the group $\GL_2(\F_2)$ is isomorphic to $S_3$, so its subgroup structure is well-known. Assume now that $\ell>2$.
Recall that a subgroup $G$ of $\GL_2(\F_\ell)$ is said to be \textit{Borel} if it is conjugated to the subgroup of upper-triangular matrices, and is said to be \textit{exceptional} if its image in $\operatorname{PGL}_2(\mathbb{F}_\ell)$ is isomorphic to $A_4, S_4$ or $A_5$. Also recall the definition of Cartan subgroups from the previous section.

\begin{theorem}[{Dickson's classification, cf.~\cite[§2]{Serre}}]\label{thm:Dickson}
Let $\ell>2$ be a prime number and $G$ be a subgroup of $\operatorname{GL}_2(\mathbb{F}_\ell)$.
\begin{itemize}
\item If $\ell$ divides the order of $G$, then $G$ either contains $\operatorname{SL}_2(\mathbb{F}_\ell)$ or is contained in a Borel subgroup.
\item If $\ell$ does not divide the order of $G$, then $G$ is contained in the normaliser of a (split or nonsplit) Cartan subgroup or in an exceptional group.
\end{itemize}
\end{theorem}

To handle the profinite groups that arise as Galois representations attached to elliptic curves we will find it useful to employ a notion first introduced by Serre \cite[IV-25]{MR1484415}. 
We say that a non-abelian
finite simple group $\Sigma$ \textit{occurs} in the profinite group $Y$ if there exist a closed subgroup $Y_1$ of $Y$ and an open normal subgroup $Y_2$ of $Y_1$ such that $\Sigma \cong Y_1/Y_2$. We notice in particular that $\PSL_2(\F_\ell)$ occurs in $\GL_2(\Z_\ell)$. We will also need the following fact: for every exact sequence $1 \to N \to G \to G/N \to 1$ and every non-abelian finite simple group $\Sigma$, if $\Sigma$ occurs in $G$ then it occurs in at least one of $N$ and $G/N$ (and conversely), see again \cite[IV-25]{MR1484415}. 

\subsection{Galois representations and torsion fields of elliptic curves}\label{subsect:NotationTorsion}

Let $K$ be a number field and $E/K$ be a fixed elliptic curve. We will say that $E$ is \textit{non-CM} if $\operatorname{End}_{\overline K}(E)$ is $\mathbb{Z}$, or equivalently, if $E$ does not have CM over $\overline K$. We will denote by $E_{\tors}$ the group of all torsion points in $E(\overline{K})$.
Consider, for each positive integer $N$, the natural Galois representation
\[
\rho_{N} : \Gal(\overline{K} \mid K) \to \Aut(E[N])
\]
afforded by the $N$-torsion points of $E(\overline{K})$. We will often assume that a basis of the free $\mathbb{Z}/N\mathbb{Z}$-module $E[N]$ has been fixed, and therefore regard the image $G_N$ of $\rho_{N}$ as a subgroup of $\GL_2(\Z/N\Z)$.

We denote by $K_{\ell^n}$ the field fixed by the kernel of $\rho_{\ell^n}$, or equivalently the Galois extension of $K$ generated by the coordinates of all $\ell^n$-torsion points of $E$.
By passing to the limit in $n$ we also obtain the group $G_{\ell^\infty} = \Gal\left(K(E[\ell^\infty]) \mid K \right)$, which we consider as a subgroup of $\GL_2(\Z_\ell)$, and the corresponding fixed field $K_{\ell^\infty} = \bigcup_{n \geq 1} K_{\ell^n}$. Finally, we also denote by $K_\infty$ the field generated by the various $K_{\ell^\infty}$ as $\ell$ varies. One can also define the adelic Tate module $TE := \varprojlim_{N} E[N]$, isomorphic to $\hat{\Z}^2$, and the adelic Galois representation $\rho_\infty : \Gal(\overline{K} \mid K) \to \Aut( TE )$. The Galois group $\Gal(K_\infty \mid K)$ is then isomorphic to the image $G_\infty$ of $\rho_\infty$ (hence to the inverse limit $\varprojlim_N \operatorname{Im} \rho_N$), and may be considered -- up to the choice of an isomorphism $E_{\tors} \cong \hat{\Z}^2$ -- as a subgroup of $\GL_2(\hat{\Z})$.
Finally we remark that, since all the representations $\rho_N$ are continuous and $\Gal(\overline{K} \mid K)$ is a compact topological group, all the groups just introduced are compact, and therefore closed in their respective ambient spaces.

\subsection{Modulo \texorpdfstring{$\ell$}{l} Galois representations of elliptic curves over \texorpdfstring{$\mathbb{Q}$}{Q}}
Our focus will be on elliptic curves defined over the field of rational numbers. The Galois representations attached to such curves have been studied extensively, and a number of powerful results on their possible images have been proven.
We will in particular need to rely on a famous theorem of Mazur concerning the degrees of cyclic isogenies of elliptic curves defined over $\mathbb{Q}$. To state it, let
\[
\mathcal{T}_0:=\set{p\text{ prime }\mid p\leq 17}\cup\{37\}.
\]
\begin{theorem}[{\cite[Theorem 1]{Mazur}}]
\label{thm:Mazur}
Let $p$ be a prime number and $E/\Q$ be an elliptic curve, and assume that $E$ has a $\mathbb{Q}$-rational subgroup of order $p$. Then $p\in \mathcal T_0\cup \set{19,43,67,163}$. If $E$ does not have CM over $\overline \Q$, then $p\in\mathcal{T}_0$.
\end{theorem}

%% file: 3-ScalarsInImage.tex
\section{Scalars in the image of Galois representations}
\label{section:scalarsInImage}
Let $E$ be an elliptic curve over a number field $K$ and let $\ell$ be a prime number.
Our purpose in this section is to study the intersection $G_{\ell^\infty} \cap \mathbb{Z}_\ell^\times \cdot \Id$, that is, the subgroup of scalar matrices in the image of the $\ell$-adic Galois representation attached to $E/K$. We will focus mostly, but not exclusively, on the case $K=\Q$. The main result is Theorem \ref{thm:ScalarsImageGalois}, which -- for each prime $\ell$ -- describes a subgroup of $\mathbb{Z}_\ell^\times \cdot \Id$ that is guaranteed to be contained in $G_{\ell^\infty}$ for all non-CM elliptic curves over $\Q$ (see also Proposition \ref{prop:ScalarsInImageCM} for the CM case). To simplify the notation, we will often identify $(\Z/\ell^n\Z)^\times$ (resp.~$\Z_\ell^\times$) with the subgroup $(\Z/\ell^n\Z)^\times \cdot \Id$ (resp.~$\Z_\ell^\times \cdot \operatorname{Id}$) of $\GL_2(\Z/\ell^n\Z)$ (resp.~$\GL_2(\Z_\ell)$).

Since it helps understanding the relevance of the criteria in the next subsection, we briefly contextualise the group-theoretic properties we are going to consider in terms of the Galois representations attached to elliptic curves over $\Q$. Let $E/\Q$ be an elliptic curve and let $G_{\ell^\infty}$ (respectively $G_\ell$) be the image of the corresponding $\ell$-adic (respectively mod $\ell$) Galois representation. To begin with, one has $\det(G_{\ell^\infty})=\mathbb{Z}_\ell^\times$, because for $\sigma \in \Gal(\overline{\Q} \mid \Q)$ the determinant of $\rho_{\ell^\infty}(\sigma)$ is simply $\chi_{\ell^\infty}(\sigma)$, and it is well-known that the $\ell$-adic cyclotomic character $\chi_{\ell^\infty}$ is surjective. Moreover, when $E$ is non-CM and $\ell \not \in \mathcal{T}_0$, by Theorem \ref{thm:Mazur} we know that $G_\ell$ acts irreducibly on $E[\ell]$; in particular, this holds for all $\ell > 37$. We prove in Lemma \ref{lemma:SuperLift} below that if $G_\ell$ acts irreducibly on $E[\ell]$ and $\ell \mid \#G_\ell$ then $G_{\ell^\infty}=\GL_2(\Z_\ell)$, so the most interesting case (for $\ell$ large) is $\ell \nmid \#G_\ell$. In this case \cite[Proposition 1.13]{PossibleImages} (or equivalently \cite[Appendix B]{PedroSamuel}) shows that (up to conjugacy) there are only two possibilities for $G_{\ell}$, namely a non-split Cartan subgroup or the unique index-3 subgroup thereof. These are therefore the most interesting situations, and are explored in Corollary \ref{cor:AllScalars}.
Finally, notice that the image of a complex conjugation in $G_{\ell^\infty}$ is a matrix of order $2$ with determinant $-1$, so -- up to conjugation -- when $\ell > 2$ we may assume that it is $\matr 0110$. This explains the relevance of this specific matrix for the statement of Proposition \ref{proposition:containsAllScalars}.

\subsection{Group-theoretic criteria}\label{subsect:GroupTheoreticCriteria}
In this section we establish several criteria that guarantee that a closed subgroup $G$ of $\GL_2(\Z_\ell)$ contains an (explicit) open subgroup of $\Z_\ell^\times$. A further result of the same kind, whose proof is however more involved, is stated and proved in Appendix \ref{sect:pAdicAppendix}.
The criteria in this section will be expressed in terms of $G_\ell$, the image of $G$ under reduction modulo $\ell$. More generally, we will employ the following notation:

\medskip

\noindent\textbf{Notation.}
Let $G$ be a subgroup of $\GL_2(\Z_\ell)$. We denote by $G_{\ell^n}$ the image of $G$ under the reduction map $\GL_2(\Z_\ell) \to \GL_2(\Z/\ell^n\Z)$.

\begin{lemma}
  \label{lemma:TeichLifts}
  Let $\ell$ be a prime and let $g\in \GL_2(\Z_\ell)$ be such that
  $g\equiv \lambda\Id\pmod\ell$ for some $\lambda\in\F_\ell^\times$.
  Let moreover $\tilde\lambda\in \Z_\ell^\times$ be the Teichm\"uller lift
  of $\lambda$. Then the sequence $\{g^{\ell^n}\}_{n\geq 1}$ converges to
  $\tilde\lambda\Id\in \GL_2(\Z_\ell)$.
\end{lemma}
\begin{proof}
  By Lemma \ref{lemma:TeichSplit} we can write $g=\tilde\lambda h$, where
  $h=\Id+\ell h_1\in\GL_2(\Z_\ell)$ is congruent to the identity modulo $\ell$. Then for
  any $n\geq 1$ we have $g^{\ell^n}=\tilde\lambda^{\ell^n}h^{\ell^n}=
  \tilde\lambda h^{\ell^n}$. By Lemma \ref{lemma:ellToTheN} we have that
  $v_\ell((\Id+\ell h_1)^{\ell^n}-\Id)>n$ for every $n\geq 0$. This means that
  the sequence $\{h^{\ell^n}\}_{n\geq 1}$ converges to $\Id$, hence
  $\set{g^{\ell^n}}_{n\geq 1}$ converges to $\tilde\lambda\Id$.
\end{proof}

\begin{corollary}\label{cor:LiftingScalars}
Let $\ell$ be a prime and let $G$ be a closed subgroup of $\GL_2(\mathbb{Z}_\ell)$. Suppose that the image of $G$ in $\GL_2(\mathbb{F}_\ell)$ contains $\lambda \operatorname{Id}$ for some $\lambda \in \mathbb{F}_\ell^\times$: then $G$ contains $\tilde{\lambda} \operatorname{Id}$.
\end{corollary}
\begin{proof}
Let $g \in G$ reduce to $\lambda \operatorname{Id}$ modulo $\ell$. By the previous lemma the sequence $\{g^{\ell^n}\}$ converges to $\tilde{\lambda} \operatorname{Id}$, so this is an element of $G$ since by assumption $G$ is closed.
\end{proof}

The following result can be found in \cite[Lemma 2.5]{zywyna-bounds}, but we
include the proof here for ease of reference.

\begin{lemma}
  \label{lemma:determinantCongruentOne}
  Let $n$ be a positive integer, let $\ell>2$ be a prime and let $G$ be a closed
  subgroup of $\GL_2(\Z_\ell)$. Let
  \begin{align*}
    H_n:=\set{g\in G\mid g\equiv\Id\pmod {\ell^n}}.
  \end{align*}
  If $\det(G)=\Z_\ell^\times$ and $\ell \nmid \#G_\ell$,
  then $\det(H_n)=1+\ell^n\Z_\ell$.
\end{lemma}
\begin{proof}
  Clearly $\det(H_n)\subseteq 1+\ell^n\Z_\ell$, so we only need to prove the
  other inclusion.
  Since $\det(G)=\Z_\ell^\times$ there is $g\in G$ such that $\det(g)=1+\ell$.
  Then by Lemma \ref{lemma:ellToTheN} the element $h:=g^{\ell^{n-1}\cdot\#G_1}$
satisfies $h\equiv \Id\pmod{\ell^n}$, so it belongs to $H_n$. Moreover
  \begin{align*}
    \det(h)=(1+\ell)^{\ell^{n-1}\cdot \#G_1}\equiv
            1+\#G_1\ell^{n}\pmod{\ell^{n+1}}
  \end{align*}
  and since $\ell\nmid \#G_\ell$ we have $v_\ell(\det(h)-1)=n$. By Lemma
  \ref{lemma:ScalarsValuationOne} we conclude that $\det(H)$ contains
  $1+\ell^n\Z_\ell$.
\end{proof}

We now come to our criterion for the existence of scalars in $G$ when $\ell \nmid \#G_\ell$.

\begin{proposition} \label{proposition:containsAllScalars}
  Let $\ell>2$ be a prime and $G$ be a closed subgroup of $\GL_2(\Z_\ell)$ 
  such that $\det G=\Z_\ell^\times$. Assume that $G_\ell$ contains
  $\tau=\matr0110$ and that $\ell \nmid \#G_\ell$.
\begin{enumerate}
\item Suppose that $G_\ell$ contains an element $u$ for which one the following holds:
\begin{enumerate} 
\item $u$ anti-commutes with $\tau$, that is, $u\tau = - \tau u$;
\item there exists $\varepsilon \in \F_\ell^\times \setminus \{1\}$ such that, for all antidiagonal matrices $A=\begin{pmatrix}
0 & x \\
y & 0
\end{pmatrix}$, we have $uAu^{-1} = \begin{pmatrix}
0 & \varepsilon x \\
\varepsilon^{-1} y  & 0
\end{pmatrix}$.
\end{enumerate}
Then $G$ contains $1+\ell\Z_\ell$.
\item Suppose that one of the assumptions of (1) holds, and that moreover $G_\ell$ contains $\F_\ell^\times$. Then $G$ contains $\Z_\ell^\times$.
\end{enumerate}
\end{proposition}

\begin{remark}\label{rmk:GoodElements}
It is immediate to check that the following elements of $\GL_2(\F_\ell)$ have the property required to apply part (1):
\begin{enumerate}
\item[(1a)] $u=\begin{pmatrix}
a & b \\ -b & -a
\end{pmatrix}$, where $a,b \in \F_\ell$ are such that $\det(u)=b^2-a^2 \neq 0$.
\item[(1b)] $u=\begin{pmatrix}
a & 0 \\ 0 & b
\end{pmatrix}$ with $a,b \in \F_\ell^\times$, $a \neq b$.
\end{enumerate}
\end{remark}

\begin{proof} By Lemma \ref{lemma:ScalarsValuationOne} the element $1+\ell$ generates $1+\ell \Z_\ell$, so it suffices to prove that $(1+\ell) \Id$ is in $G$.
As $G$ is closed, it suffices to show that $(1+\ell)\Id$ is in $G_{\ell^n}$ for every $n\geq 1$. We prove this by induction. For $n=1$ the statement holds trivially, so
  assume that $(1+\ell)\Id$ belongs to $G_{\ell^n}$ and let
  $C=(1+\ell)\Id+\ell^nB$ be a lift of this element to $G_{\ell^{n+1}}$, which
  exists because the map $G_{\ell^{n+1}}\to G_{\ell^n}$ is surjective. Notice that we may consider $B$ as an element of $\operatorname{Mat}_{2 \times 2}(\mathbb{F}_\ell)$. In addition, if $n=1$, thanks to Lemma \ref{lemma:determinantCongruentOne} we may assume that
  $\det(C)\not\equiv 1\pmod {\ell^2}$, and consequently that
  $\tr(B)\not\equiv -2\pmod\ell$. If $\tilde\tau$ is any lift of $\tau$ to
  $G_{\ell^{n+1}}$, the element
  \begin{align*}
   C' := C\tilde\tau C\tilde\tau^{-1}&=\left((1+\ell)\Id+\ell^nB\right)
    \left((1+\ell)\Id+\ell^n\tilde\tau B\tilde\tau^{-1}\right)\\
                    &=(1+\ell)^2\Id+(1+\ell)\ell^n(B+\tilde\tau B\tilde\tau^{-1})
                    +\ell^{2n}B\tilde\tau B\tilde\tau^{-1}\\
                    &\equiv (1+\ell)^2\Id+\ell^n(B+\tilde\tau B\tilde\tau^{-1})
                    \pmod{\ell^{n+1}}
  \end{align*}
is in $G_{\ell^{n+1}}$. Notice that
  $D:=B+\tau B\tau^{-1}$ is congruent to $\matr abba$ modulo $\ell$, where $a=\tr(B)$ and $b\in \F_\ell$.
  
\begin{itemize}
\item Suppose that $G_\ell$ contains an element $u$ as in part (1a). Then
\begin{align*}
u D u^{-1}\equiv\matr a{-b}{-b}a\pmod \ell\,.
\end{align*}
If $\tilde u\in G_{\ell^{n+1}}$ is a lift of $u$, the group $G_{\ell^{n+1}}$ contains
  \begin{align*}
    C'\tilde{u} C'\tilde{u}^{-1}& \equiv \left((1+\ell)^2\Id+\ell^nD\right)
    \left((1+\ell)^2\Id+\ell^n\tilde u D\tilde u^{-1}\right)\\
    &\equiv \left((1+\ell)^2\Id+\ell^n\matr abba\right)
      \left((1+\ell)^2\Id+ \ell^n\matr a{-b}{-b}a\right)\\
    &\equiv (1+\ell)^4\Id+2a\ell^n\Id\pmod{\ell^{n+1}}
  \end{align*}
  which is a scalar matrix congruent to $1+4\ell$ modulo $\ell^2$ if 
  $n>1$ or to $1+2\ell(2+a)$ if $n=1$. 
  
\item Suppose that $G_{\ell}$ contains an element $u$ as in part (1b). Then we have
\[
D_k := u^kDu^{-k} = \begin{pmatrix}
a & b\varepsilon^k \\
b\varepsilon^{-k} & a
\end{pmatrix}.
\]
Letting $\tilde{u}$ be a lift of $u$ to $G_{\ell^{n+1}}$ we obtain that for every non-negative integer $k$ the group $G_{\ell^{n+1}}$ contains
\[
\tilde{u}^k C' \tilde{u}^{-k} = (1+\ell)^2 \operatorname{Id} + \ell^n D_k.
\]
Thus, using the fact that $\sum_{k=0}^{\ell-2} \varepsilon^k = \frac{\varepsilon^{\ell-1}-1}{\varepsilon-1} = 0$, we see that $G_{\ell^{n+1}}$ also contains

\[
\begin{aligned}
\prod_{k=0}^{\ell-2} \tilde{u}^k C' \tilde{u}^{-k} & \equiv \prod_{k=0}^{\ell-2} \left( (1+\ell)^2 \operatorname{Id} + \ell^n D_k \right) \pmod{\ell^{n+1}} \\
& \equiv (1+\ell)^{2(\ell-1)} \operatorname{Id} + \ell^n(1+\ell)^{2(\ell-2)} \sum_{k=0}^{\ell-2} D_k  \pmod{\ell^{n+1}}\\
& \equiv
(1+\ell)^{2(\ell-1)} \operatorname{Id} - \ell^n \begin{pmatrix}
a & 0 \\
0 & a
\end{pmatrix} \pmod{\ell^{n+1}},
\end{aligned}
\]
  which is a scalar matrix congruent to $1-2\ell$ modulo $\ell^2$ if 
  $n>1$ or to $1-(2+a)\ell$ if $n=1$.
\end{itemize}  
In any case, using our assumption that $a=\tr(B) \not \equiv -2 \pmod{\ell}$ if $n=1$, we
see that $G_{\ell^{n+1}}$ contains a scalar matrix $\lambda \operatorname{Id}$ with $v_\ell(\lambda-1)=1$. 
  We can now apply Lemma \ref{lemma:ScalarsValuationOne} to the subgroup of
  $\Z_\ell^\times$ given by the inverse image of $G_{\ell^{n+1}}\cap(\Z/\ell^{n+1}\Z)^\times$
  under the natural projection, and we conclude that $(1+\ell)\Id\in G_{\ell^{n+1}}$ as desired.
    
Finally, if $\F_\ell^\times$ is contained in $G_{\ell}$, Lemma
  \ref{lemma:TeichLifts} shows that $G$ contains a Teichm\"uller lift of
  every element of $\F_\ell^\times$. By Lemma \ref{lemma:TeichSplit} this is
  enough to conclude that $G$ contains $\Z^\times_\ell$.
\end{proof}

\begin{lemma}\label{lemma:SuperLift}
Let $\ell \geq 5$ be a prime number and $G$ be a closed subgroup of $\GL_2(\Z_\ell)$. Suppose that $\det(G) = \Z_\ell^\times$. If $\ell \mid \#G_\ell$ and $G_\ell$ acts irreducibly on $\F_\ell^2$, then $G=\GL_2(\Z_\ell)$.
\end{lemma}
\begin{proof}
Since $\ell \mid \#G_\ell$, the classification of maximal subgroups of $\GL_2(\mathbb{F}_\ell)$ (Theorem \ref{thm:Dickson}) shows that either $G_\ell$ is contained in a Borel subgroup of $\GL_2(\F_\ell)$, or $G_\ell$ contains $\SL_2(\mathbb{F}_\ell)$. However, any subgroup of a Borel acts reducibly on $\F_\ell^2$ by definition, hence we see that $G_\ell$ contains $\SL_2(\mathbb{F}_\ell)$. By a lemma due to Serre (see \cite[IV-23, Lemme 3]{MR1484415} and \cite[Lemma 3.15]{MR3437765} for this exact version), this implies that $G$ contains $\SL_2(\mathbb{Z}_\ell)$. From $\det(G)=\Z_\ell^\times$ we then obtain $G=\GL_2(\Z_\ell)$ as desired.
\end{proof}

\begin{corollary}\label{cor:AllScalars}
Let $\ell > 2$ be a prime and let $G$ be a closed subgroup of $\GL_2(\mathbb{Z}_\ell)$ with $\det(G) = \mathbb{Z}_\ell^\times$. Suppose that (at least) one of the following holds:
\begin{enumerate}
\item $G_\ell \subseteq \GL_2(\mathbb{F}_\ell)$ contains (up to conjugacy) the normaliser of a split or non-split Cartan, and if $\ell \mid \#G_\ell$ then $\ell \neq 3$.
\item $\ell \equiv 2 \pmod 3$, and $G_\ell \subset \GL_2(\mathbb{F}_\ell)$ contains (up to conjugacy) the subgroup of cubes in the normaliser of a non-split Cartan.
\end{enumerate} 
Then $G$ contains $\mathbb{Z}_\ell^\times$.
\end{corollary}
\begin{proof}
Suppose first that $\ell \mid \#G_\ell$ (hence in particular $\ell>3$). The normaliser of a (split or non-split) Cartan, or an index-3 subgroup of a non-split Cartan, acts irreducibly on $\F_\ell^2$, so Lemma \ref{lemma:SuperLift} implies $G=\GL_2(\Z_\ell)$, which in particular contains $\Z_\ell^\times$.

Suppose on the other hand that $\ell \nmid \#G_\ell$. Notice that -- since the scalar matrices are contained in the centre of $\GL_2(\Z_\ell)$ -- the conclusion of Proposition \ref{proposition:containsAllScalars} is invariant under a change of basis for $\Z_\ell^2$, so it suffices to check that the group $G$ satisfies the hypotheses of Proposition \ref{proposition:containsAllScalars} after a suitable change of basis.
\begin{enumerate}
\item By what we already remarked, and up to conjugation in $\GL_2(\Z_\ell)$, we may assume that $G_\ell$ contains the group $N_\ell^*(\varepsilon)$ described in Remark \ref{remark:newCartanModel}, or an index-3 subgroup thereof. The explicit description shows that every group of the form $N_\ell^*(\varepsilon)$ contains $\matr 0110$; since this matrix is equal to its cube, $\matr 0110$ is also contained in the subgroup of cubes in $N_\ell^*(\varepsilon)$.

The normaliser of a split Cartan subgroup contains all anti-diagonal matrices, hence in particular it contains $u=\matr 0{-1}10$. The normaliser of a non-split Cartan contains $u=\matr{\varepsilon}{-1}{1}{-\varepsilon}$. Finally, the subgroup of cubes of such a normaliser contains $\matr{\varepsilon}{-1}{1}{-\varepsilon}^3=(\varepsilon^2-1) \matr{\varepsilon}{-1}{1}{-\varepsilon}$. In all cases we have thus shown that $G_\ell$ contains an element of the form required to apply Proposition \ref{proposition:containsAllScalars} (1), see Remark \ref{rmk:GoodElements}.
\item 
As for hypothesis (2) of Proposition \ref{proposition:containsAllScalars}, observe that all scalar matrices are contained in the normaliser of every (split or non-split) Cartan subgroup of $\GL_2(\F_\ell)$. When $\ell \equiv 2 \pmod{3}$ they are also contained in the subgroup of cubes of a non-split Cartan: indeed, in this case $x \mapsto x^3$ is an automorphism of $\mathbb{F}_\ell^\times$, so every scalar matrix is a cube.
\end{enumerate}

\end{proof}

\subsection{Scalars in the presence of an isogeny}
We now specialise to the case of $G = G_{\ell^\infty} \subseteq \GL_2(\Z_\ell)$ being the image of the $\ell$-adic representation attached to an elliptic curve $E/\mathbb{Q}$. Our aim is again to prove that $G$ contains an (explicitly identifiable) subgroup of $\Z_\ell^\times$.
 We begin by considering the case when $\ell \geq 7$ and $E$ admits an isogeny of degree $\ell$ defined over $\mathbb{Q}$. The relevant results are essentially already in the literature, and in this short section we reformulate them in the form needed for our applications.

\begin{definition}[{\cite[Definition 1.1]{Isogeny7}}]
An elliptic curve $E$ over $\mathbb{Q}$ is called \textbf{$\ell$-exceptional}, where $\ell$ is a prime, if $E$ has an isogeny of degree $\ell$ defined over $\mathbb{Q}$ and $G_{\ell^\infty}$ does not contain a Sylow pro-$\ell$ subgroup of $\GL_2(\mathbb{Z}_\ell)$.
\end{definition}
Combining \cite[Theorem 1]{Isogeny} with \cite[Remark 4.2.1]{Isogeny} and \cite[Theorem 5.5]{Isogeny7} one obtains:
\begin{theorem}\label{thm:Isogeny}
Let $\ell \geq 7$ be a prime. There are no non-CM $\ell$-exceptional elliptic curves defined over $\mathbb{Q}$.
\end{theorem}

For the case $\ell=5$ we instead rely on the following result:
\begin{theorem}[{\cite[Theorem 2]{Isogeny}}]\label{thm:Isogeny5}
Let $E/\Q$ be a non-CM elliptic curve.
Suppose that $E$ has an isogeny of degree 5 defined over $\Q$. If none of the elliptic curves in the $\Q$-isogeny class of $E$ has two independent isogenies of degree 5, then $E$ is not $5$-exceptional. Otherwise, the index $[\GL_2(\Z_5) : G_{5^\infty}]$ is divisible by 5, but not by $25$.
\end{theorem}

\begin{corollary}\label{cor:ScalarsReducibleCase}
Let $E/\mathbb{Q}$ be a non-CM elliptic curve, let $\ell \geq 5$ be a prime number, and suppose that the Galois module $E[\ell]$ is reducible. Then $G_{\ell^\infty}$ contains $1+\ell \mathbb{Z}_\ell$.
\end{corollary}
\begin{proof}
A specific Sylow pro-$\ell$ subgroup $S$ of $\GL_2(\mathbb{Z}_\ell)$ is given by 
\[
S = \set{ M = \matr abcd \in \GL_2(\mathbb{Z}_\ell) \bigm\vert a \equiv d \equiv 1 \pmod \ell, \; c \equiv 0 \pmod{\ell}  }.
\]
It is clear that $1+\ell \mathbb{Z}_\ell$ is contained in $S$. However, since all the pro-$\ell$ Sylow subgroups of $\GL_2(\mathbb{Z}_\ell)$ are conjugated to each other and $1+\ell \mathbb{Z}_\ell$ lies in the center of $\GL_2(\mathbb{Z}_\ell)$ (hence is stable under conjugation), it follows that $1+\ell \mathbb{Z}_\ell$ is contained in \textit{all} the Sylow pro-$\ell$ subgroups of $\GL_2(\mathbb{Z}_\ell)$. For $\ell \geq 7$ the statement then becomes a direct consequence of Theorem \ref{thm:Isogeny}. For $\ell=5$ the claim similarly follows from Theorem \ref{thm:Isogeny5} if no elliptic curve in the $\Q$-isogeny class of $E$ admits two independent $5$-isogenies. To treat this last case, observe that the intersection $G_{\ell^\infty} \cap \Z_\ell^\times$ is the same for all the elliptic curves in a given $\Q$-isogeny class (see e.g.~\cite[§2.4]{Isogeny}), so we may assume that $E$ admits two independent $5$-isogenies defined over $\Q$. In particular, the Galois module $E[5]$ decomposes as the direct sum of two 1-dimensional modules, which implies that in a suitable basis $G_5$ consists of diagonal matrices. Hence $[\GL_2(\F_5) : G_5]$ is divisible by $5$, and on the other hand $25 \nmid [\GL_2(\Z_5) : G_{5^\infty}]$ by Theorem \ref{thm:Isogeny5} again. It follows immediately that $\ker(\GL_2(\Z_5) \to \GL_2(\F_5))$, which is a pro-$5$ group, is entirely contained in $G_{5^\infty}$, hence in particular that $1+5\Z_5 \subseteq G_{5^\infty}$, as desired.
\end{proof}

\subsection{The $3$-adic case}\label{subsect:3adicScalars}
Let $E/\Q$ be a non-CM elliptic curve. Relying on the group-theoretic results of Appendix \ref{sect:pAdicAppendix} we now prove that the $3$-adic Galois representation attached to $E$ contains all scalars congruent to 1 modulo 27. We treat separately the two cases when the Galois module $E[3]$ is respectively irreducible or reducible.

\subsubsection{Irreducible case} When $E[3]$ is irreducible for the Galois action, it is not hard to prove that $G_{3^\infty}$ contains all scalars:
\begin{proposition}\label{prop:3adicIrreducible}
Suppose $E[3]$ is an irreducible Galois module. Then $G_{3^\infty}$ contains $\Z_3^\times$.
\end{proposition}
\begin{proof}
Up to conjugation, we can assume that $G_3$ contains $\begin{pmatrix}
1 & 0 \\ 0 & -1
\end{pmatrix}$ (the image of complex conjugation). A short direct computation shows that (up to conjugacy) there are only 3 possibilities for $G_3$, namely $\GL_2(\F_3)$, a 2-Sylow subgroup, or the group $H := \langle \begin{pmatrix}
0 & 1 \\ 1 & 0
\end{pmatrix}, \begin{pmatrix}
0 & -1 \\
1 & 0
\end{pmatrix} \rangle$ of order 8. In particular, in all cases we may assume that $H \subseteq G_3$. 
The hypotheses of Proposition \ref{proposition:containsAllScalars} (2) are then satisfied, hence $G_{3^\infty}$ contains $\Z_3^\times$.
\end{proof}

\subsubsection{Reducible case}
We now consider the much harder case when $E[3]$ is reducible under the Galois action.
Our analysis is based on the purely group-theoretic Proposition \ref{prop:padicscalars}. To motivate the hypotheses that appear in its statement, we consider a non-CM elliptic curve $E/\Q$ for which the Galois module $E[3]$ is reducible, and denote as usual by $G_{3^n}$ the image of the modulo-$3^n$ representation attached to $E/\Q$ and by $G_{3^\infty}$ the image of the $3$-adic representation. The following hold:
\begin{enumerate}
\item Any elliptic curve $\tilde{E}/\Q$ that is $\Q$-isogenous to $E$ gives rise to a $3$-adic Galois image $\tilde{G}_{3^\infty}$ for which $G_{3^\infty} \cap \Z_3^\times = \tilde{G}_{3^\infty} \cap \Z_{3}^\times$ (notice that this equality is independent of the choice of basis for $T_3 E, T_3 \tilde{E}$), see for example \cite[§2.4]{Isogeny}. For all such curves $\tilde{E}/\Q$, the Galois module $\tilde{E}[3]$ is clearly reducible, and at least one $\tilde{E}$ of this form does not admit two independent cyclic isogenies of degree 3 defined over $\Q$. Hence, up to replacing $E$ with $\tilde{E}$, we may assume that $G_3$ is contained (up to conjugacy) in a Borel subgroup and that $G_3$ only fixes one nontrivial $\mathbb{F}_3$-subspace of $E[3]$. This implies $3 \mid \#G_3$.
\item $G_{27}$ acts on $E[27]$ without fixing any cyclic subgroup of order 27. Indeed, the three rational points on $X_0(27)$ are two cusps and a single non-cuspidal point corresponding to a CM elliptic curve \cite[p.~229]{MR0337974}.
\item $\det(G_{3^\infty})=\Z_3^\times$: as already discussed, this follows from the surjectivity of the $3$-adic cyclotomic character.
\item $G_{3^\infty}$ contains the image of (any) complex conjugation, which is an element $c$ of order 2 with determinant $-1$.
\end{enumerate}

We now check that these information are sufficient to apply Proposition \ref{prop:padicscalars}. Up to a change of basis, we may assume that the element $c \in G_{3^\infty}$ is represented by the matrix $C=\matr 100{-1}$. This easily implies that $G_3$ is contained in the Borel of upper- or lower-triangular matrices (see also Remark \ref{rmk:MatrixD}).
Take now $H$ to be pro-$3$ Sylow subgroup of $G_{3^\infty}$ (which is normal, hence unique: it is the inverse image in $G_{3^\infty}$ of the $3$-Sylow of $G_3$, which is easily checked to be normal). We claim that this group satisfies all the assumptions of Proposition \ref{prop:padicscalars} with $p=3$ and $k=3$. Hypothesis (1) is satisfied by (1) above. 
 Hypothesis (3) is clear from the equality $\det(G_{3^\infty})=\Z_3^\times$, and (4) follows from the fact that $C \in G_{3^\infty}$ and $H$ is normal in $G_{3^\infty}$. As for (2), 
recall that $G_3$ is contained in the upper- or lower-triangular Borel subgroup, and this implies easily that $G_{3^\infty}$ is generated by $H$, $C$, and possibly $-\Id$. Since both $C$ and $-\Id$ are diagonal, we see that if $H_{3^3}$ is upper- or lower- triangular, then so is $G_{3^3}$, contradiction, because we know that $E$ does not admit any cyclic 27-isogeny defined over $\Q$. Hence from Proposition \ref{prop:padicscalars} we obtain:
\begin{proposition}
Let $E/\Q$ be a non-CM elliptic curve for which $E[3]$ is a reducible Galois module. Then $G_{3^\infty}$ contains all scalars congruent to $1$ modulo 27.
\end{proposition}
Combining this result with Proposition \ref{prop:3adicIrreducible} we have then proved:
\begin{corollary}\label{cor:3adic}
Let $E/\Q$ be a non-CM elliptic curve. The group $G_{3^\infty}$ contains all scalars congruent to 1 modulo 27.
\end{corollary}

\begin{remark} \label{rem:RouseExponent3}
	J. Rouse informed us that 
	for
	every non-CM elliptic curve over $\Q$ with a rational $3$-isogeny the group $G_{3^\infty}$
	contains all scalars congruent to $1$ modulo $9$ (hence, by Proposition \ref{prop:3adicIrreducible}, the same holds for every non-CM $E/\mathbb{Q}$).
	This will be shown in upcoming work by Rouse, Sutherland and Zureick-Brown.
	Their proof relies on the explicit determination of the rational points of suitable modular curves. As pointed out in the introduction,
	we think our approach -- that derives the result from properties of
	isogenies (hence relying only on the more well-studied modular curves $X_0(N)$) -- has the advantage of being easier to extend to number fields
	different from $\Q$.
\end{remark}

\subsection{Main theorem}
We are now ready to prove our uniform result for scalars in the image of Galois representations:
\begin{theorem}\label{thm:ScalarsImageGalois}
Let $E$ be a non-CM elliptic curve over $\mathbb{Q}$ and let $\ell$ be a prime number. Define

\[
s_\ell := \begin{cases}
4, \text{ if } \ell=2 \\
3, \text{ if } \ell=3 \\
1, \text{ if } \ell=5, 7, 11, 13, 17, 37 \\
0, \text{ if } \ell \geq 19 \text{ and } \ell \neq 37
\end{cases}
\]
The image $G_{\ell^\infty}$ of the $\ell$-adic Galois representation attached to $E/\mathbb{Q}$ contains all scalars congruent to $1$ modulo $\ell^{s_\ell}$.
\end{theorem}
\begin{proof}
For $\ell=2$ and $\ell=3$ the theorem follows from the results of \cite{MR3500996} and Corollary \ref{cor:3adic} respectively. We may therefore assume $\ell \geq 5$.
We distinguish several cases:
\begin{enumerate}
\item the $G_\ell$-module $E[\ell]$ is reducible. The claim follows from Corollary \ref{cor:ScalarsReducibleCase}.

\item the $G_\ell$-module $E[\ell]$ is irreducible and $\ell \mid \#G_\ell$. By Lemma \ref{lemma:SuperLift} we obtain $G_{\ell^\infty} = \GL_2(\Z_\ell)$, and the claim follows.

\item the $G_\ell$-module $E[\ell]$ is irreducible and $\ell \nmid \#G_\ell$. Suppose first that $\ell \geq 17$: then the claim follows from \cite[Proposition 1.13]{PossibleImages} (the exceptional $j$-invariants correspond to elliptic curves for which $G_\ell$ does not act irreducibly on $E[\ell]$, see \cite[Theorem 1.10]{PossibleImages}). For $\ell=5, 7, 11$, Theorems 1.4, 1.5 and 1.6 in \cite{PossibleImages} completely describe the possible mod-$\ell$ images $G_{\ell}$. Since $G_\ell$ acts irreducibly on $E[\ell]$ by assumption, we need to consider the following cases:
\begin{enumerate}
\item for $\ell=5$, up to conjugacy the group $G_\ell$ contains either the index-3 subgroup of a non-split Cartan or the full normaliser of a split Cartan. In both cases we may apply Corollary \ref{cor:AllScalars}.
Similarly, for $\ell=11$, up to conjugacy the only possibility is that $G_\ell$ is the full normaliser of a non-split Cartan, and again we conclude by Corollary \ref{cor:AllScalars}.

\item for $\ell=7$, up to conjugacy we have that $G_\ell$ is the normaliser of a (split or non-split) Cartan subgroup, or that it contains $\langle \begin{pmatrix}
2 & 0 \\
0 & 4
\end{pmatrix}, \begin{pmatrix}
0 & 2 \\
1 & 0
\end{pmatrix} \rangle$. The first case is handled as above. In the other case, one checks that $G_\ell$ contains $\matr 0110$, and clearly it contains $\matr 2004$, so the hypothesis of Proposition \ref{proposition:containsAllScalars} (1b) is satisfied (see Remark \ref{rmk:GoodElements}) and the claim follows.
\end{enumerate}
This only leaves the prime $\ell=13$. By \cite[§1.6]{PossibleImages}, the maximal proper subgroups of $\GL_2(\mathbb{F}_{13})$ not contained in a Borel are (up to conjugacy) the normalisers of (split and non-split) Cartan subgroups and the group 
\[
G_{S_4} = \langle \matr 2002, \matr 2003, \matr 0{-1}10, \matr 11{-1}1
 \rangle.
 \]
 The main result of \cite{MR3961086} (precisely, Theorem 1.1 and Corollary 1.3 in op.~cit.) shows that $G_{13}$ is not conjugate to a subgroup of a (split or non-split) Cartan. It remains to understand the case $G_{13} \subseteq G_{S_4}$. Consider the collection $\mathcal{C}$ of subgroups $H \subseteq G_{S_4}$ that satisfy all of the following conditions:
\begin{enumerate}
\item $\det H = \mathbb{F}_{13}^\times$;
\item $H$ contains an element $h$ with $h^2=\operatorname{Id}$ and $\tr(h)=0$;
\item the projective image $H / (H \cap \mathbb{F}_{13}^\times)$ has exponent at least 3;
\item $H$ acts irreducibly on $E[13]$.
\end{enumerate}
If $E$ is a non-CM elliptic curve over $\mathbb{Q}$ such that $G_{13}$ is contained (up to a choice of basis for $E[13]$) in $G_{S_4}$ and not contained in a Borel subgroup, then $G_{13}$ is a member of $\mathcal{C}$: (a) follows from the surjectivity of the mod-13 cyclotomic character over $\mathbb{Q}$, (b) holds because the image of complex conjugation has these properties, (c) holds by \cite[Lemma 2.4]{DAVID20112175}, and (d) is true by definition. One checks easily that all the groups $H$ in class $\mathcal{C}$ contain both $\matr 0110$ and $\matr 01{-1}0$, hence once again Proposition \ref{proposition:containsAllScalars} (1) applies to show that $1+13\mathbb{Z}_{13} \subseteq G_{13^\infty}$, as desired.
\end{enumerate}
\end{proof}

\begin{remark}
Theorem 1.1 in the very recent preprint \cite{balakrishnan2021quadratic}, combined with \cite{MR3263141}, gives the finite list of $j$-invariants of non-CM elliptic curves $E/\Q$ for which $G_{13}$ is contained (up to conjugation) in $G_{S_4}$.  For each of these elliptic curves, the image of $G_{13}$ in $\PGL_2(\F_{13})$ is isomorphic to $S_4$: while this is not necessary for our proof, it can be used to simplify the case $\ell=13$ of the previous argument.
\end{remark}

We also have a similar result in the CM case:
\begin{proposition}\label{prop:ScalarsInImageCM}
Let $E/\Q$ be an elliptic curve with CM and let $\ell$ be a prime number. Define
\[
n_\ell' =
\begin{cases}
3, \text{ if } \ell=2 \\
1, \text{ if } \ell=3, 7, 11, 19, 43, 67, 163 \\
0, \text{ if } \ell \neq 2, 3, 7, 11, 19, 43, 67, 163
\end{cases}
\]
The image $G_{\ell^\infty}$ of the $\ell$-adic Galois representation attached to $E/\mathbb{Q}$ contains all scalars congruent to $1$ modulo $\ell^{n_\ell'}$. Moreover, for $\ell \geq 5$ the image $G_{\ell^\infty}$ contains a scalar not congruent to $\pm 1 \pmod{\ell}$.
\end{proposition}
\begin{proof}
Let $K$ be the imaginary quadratic field of complex multiplication of $E$, let $\Delta_K$ be its discriminant, and let $\mathcal{O}_{K,f}$ be the endomorphism ring of $E_{\overline{\Q}}$, seen as a subring of $\mathcal{O}_K$ (here $f$ denotes the conductor of the order $\mathcal{O}_{K,f}$ in $\mathcal{O}_K$). It is well-known that there are 13 possible pairs $(K,f)$, given by $K=\mathbb{Q}(i)$ and $f=1,2$, $K=\mathbb{Q}(\zeta_3)$ and $f=1,2,3$, $K=\mathbb{Q}(\sqrt{-7})$ and $f=1,2$, and $K=\mathbb{Q}(\sqrt{-d})$ for $d=2, 11, 19, 43, 67, 163$ with $f=1$ (see for example \cite[Appendix A, §3]{MR1312368}).
If $\ell \nmid 2f \Delta_K$, then by \cite[Theorem 1.2 (4) and Theorem 1.4]{2018arXiv180902584L} the $\ell$-adic image $G_{\ell^\infty}$ contains all scalars. If $\ell \mid f\Delta_K$ and $\ell > 2$, then $G_{\ell^\infty}$ contains $\Z_\ell^{\times 2}$ by \cite[Theorem 1.5]{2018arXiv180902584L}: notice that by the above this is only possible for $\ell=3, 7, 11, 19, 43, 67, 163$, and that for $\ell \geq 7$ the group $\Z_\ell^{\times 2}$ contains scalars not congruent to $\pm 1 \pmod \ell$. Finally, for $\ell=2$ we have by \cite[Theorems 1.6, 1.7, 1.8]{2018arXiv180902584L} that $G_{2^\infty}$ contains all scalars congruent to $1$ modulo $8$.
\end{proof}

\begin{remark}
A slightly worse result can be obtained more easily (without the need to distinguish cases) by applying \cite[Theorem 1.5]{MR3766118}. 
\end{remark}

\subsection{Complements to Theorem \ref{thm:ScalarsImageGalois}}

For future use, we record here the following modest strengthening of Theorem \ref{thm:ScalarsImageGalois}:
\begin{proposition}\label{prop:SmallImprovement}
Let $E/\Q$ be a non-CM elliptic curve. 
Let $\ell \in \{13,17,37\}$. The image of the $\ell$-adic Galois representation attached to $E/\Q$ contains a scalar $\lambda$ with $v_\ell(\lambda^2-1)=0$.
\end{proposition}
\begin{proof}
By Corollary \ref{cor:LiftingScalars} it suffices to show that $G_{\ell}$ contains a scalar different from $\pm 1$.
For $\ell= 17, 37$, this follows directly from the results of \cite{PossibleImages} (specifically, Theorem 1.10 and Proposition 1.13). For $\ell=13$, by Theorem \ref{thm:Dickson} and the fact that $G_{13}$ has surjective determinant we know that $G_{13}$ satisfies one of the following:
\begin{enumerate}
\item $G_{13}=\GL_2(\F_{13})$: in this case the conclusion is obvious.
\item $G_{13}$ is contained up to conjugacy in a Borel subgroup: by \cite[Theorem 1.8]{PossibleImages}, the possible groups that arise in this way all contain a scalar different from $\pm 1$.
\item $G_{13}$ is contained up to conjugacy in the normaliser of a (split or nonsplit) Cartan subgroup: this is impossible by the main result of \cite{MR3961086}.
\item the projective image of $G_{13}$ is isomorphic to a subgroup of $S_4$ or $A_5$: the claim follows from Lemma \ref{lemma:Scalars13} below.
\end{enumerate}

\end{proof}

\begin{lemma}\label{lemma:Scalars13}
Let $G$ be a subgroup of $\GL_2(\F_{13})$ having projective image isomorphic to a subgroup of $S_4$ or $A_5$. Suppose that $\det(G)=\mathbb{F}_{13}^\times$: then $G$ contains a scalar different from $\pm \Id$.
\end{lemma}
\begin{proof}
The hypothesis implies that the cyclic group $\mathbb{F}_{13}^\times$ is a quotient of $G$, so $G$ contains an element of order $12$.
If the claim were false, the projection map $G \to \PGL_2(\mathbb{F}_{13})$ would have kernel of order at most 2. The maximal order of an element in $S_4$ is 4, and in $A_5$ is 5. It would follow that the maximal order of an element in $G$ is at most 10, contradiction.
\end{proof}

%% file: 4-ExponentH1.tex
\section{Galois cohomology of torsion points}\label{sect:CohomologyBound}

In this section we show that there exists a universal constant $e > 0$ such that, for all elliptic curves $E/\Q$ and all positive integers $M, N$ with $N \mid M$, the cohomology group $H^1( \Gal(\Q_M \mid \Q), E[N])$ is killed by multiplication by $e$ (which we denote by $[e]$). We also provide an explicit admissible value for $e$.

We begin by showing that it suffices to consider the cohomology groups $H^1( G_\infty, E[N])$.
\begin{lemma}\label{lemma:CohomologyBoundReductionToGInfty}
Let $E/K$ be an elliptic curve over a number field $K$ and let $M,N$ be positive integers with $N \mid M$. Suppose that $H^1( G_\infty, E[N])$ is killed by $[e]$: then $H^1( \Gal(K_M \mid K), E[N])$ is also killed by $[e]$.
\end{lemma}
\begin{proof}
Denote by $H$ the kernel of the natural map $G_\infty \to \Gal(K_M \mid K)$. As $H$ acts trivially on $E[N]$ by the assumption $N \mid M$, the inflation-restriction exact sequence gives an injection of $H^1( G_\infty /H , E[N])=H^1( \Gal(K_M \mid K) , E[N])$ into $H^1( G_\infty, E[N])$, and the claim follows.
\end{proof}

On the other hand, if $H^1( \Gal(K_M \mid K), E[N])$ is killed by $[e]$ for all $M$ divisible by $N$, passing to the limit in $M$ we also obtain that $[e]$ kills $H^1( G_\infty, E[N])$.
The statement we aim for is thus equivalent to saying that, for every $E/\Q$ and positive integer $N$, the group $H^1( G_\infty, E[N])$ has finite exponent dividing $e$.
Our main tool for bounding the exponent of cohomology groups is the
following lemma (see for example \cite[Lemma A.2]{MR2018998} for a proof).
\begin{lemma}[Sah's Lemma]
  \label{lemma:Sah}
  Let $G$ be a profinite group, let $M$ be a continuous $G$-module and let $g$
  be in the centre of $G$. Then the endomorphism $x\mapsto gx-x$ of $M$ induces
  the zero map on $H^1(G,M)$. In particular, if $x\mapsto gx-x$ is an
  isomorphism, then $H^1(G,M)=0$.
\end{lemma}

\begin{remark}\label{rmk:Sah}
In our applications of Lemma \ref{lemma:Sah} we will have $G \subseteq \GL_2(R)$ for a certain ring $R$ -- either $\Z_\ell$ for some prime $\ell$ or $\hat{\Z}$ -- and $M$ will be a submodule of $\left( \Q_\ell / \Z_\ell\right)^2$ or $\left( \Q/\Z \right)^2$. Notice that these objects carry a natural action of $\GL_2(\Z_\ell)$ and $\GL_2(\hat{\Z})$ respectively.
We will take $g$ to be a scalar multiple of the identity, that is, $g = \lambda \operatorname{Id}$ for some $\lambda \in R^\times$. The conclusion is then that the $R$-module 
$H^1(G,M)$ is killed by $\lambda-1$; when $R=\mathbb{Z}_\ell$, this is equivalent to saying that $H^1(G,M)$ is killed by $\ell^{v_\ell(\lambda-1)}$.
\end{remark}

Generalising the results of \cite{LawsonWutrich} we now give a uniform result on the cohomology of torsion points of elliptic curves over $\Q$ for all powers of primes.
\begin{theorem}
\label{theorem:ExpCohomologySmallPrimes}
Let $\ell$ be a prime number and let $E/\Q$ be a non-CM elliptic curve. For every $m \geq 1$, the exponent of $H^1(G_{\ell^\infty}, E[\ell^m])$ divides $\ell^{n_\ell}$, where
\begin{equation}\label{eq:n_ell}
n_\ell:=
    \begin{cases}
      3 & \text{for }\ell=2, 3,\\
      1 & \text{for }\ell=5, 7, 11,\\
      0 & \text{for }\ell\geq 13\,.
    \end{cases}
\end{equation}
\end{theorem}
\begin{proof}
For $\ell>2$ we apply Lemma \ref{lemma:Sah} (in the form of Remark \ref{rmk:Sah}) with $g=\lambda \operatorname{Id}$, where $\lambda \in \Z_\ell^\times \cap G_{\ell^\infty}$ is such that $v_\ell(\lambda-1)=n_\ell$. Note that such a $\lambda$ exists by Theorem \ref{thm:ScalarsImageGalois} and Proposition \ref{prop:SmallImprovement}.

For $\ell=2$ the proof is based on the classification of all possible 2-adic images provided by \cite{MR3500996}, and is in part computational. As $G_{2^\infty}$ is the inverse limit of the groups $G_{2^n}$, it suffices to show that for all integers $n \geq m \geq 1$ the exponent of $H^1(G_{2^n}, E[2^m])$ divides 8.
 If $G_{2^\infty}$ contains a scalar $\lambda$ with $v_2(\lambda-1) \leq 3$ the result follows immediately from Lemma \ref{lemma:Sah} as above, so let us assume that this is not the case. This leaves us with only 8 groups left, namely those with Rouse--Zureick-Brown labels X238a,
X238b,
X238c,
X238d,
X239a,
X239b,
X239c,
X239d.
All of these groups are the inverse images of their reduction modulo $2^5$ and contain $17\Id$. Let now $\xi : G_{2^n} \to E[2^m]$ be a 1-cocycle and let $\lambda \in G_{2^n}$ be the scalar $17\Id$. Notice that there is nothing to prove if $m \leq 3$, so we may assume $n \geq m \geq 4$. Reasoning as in the proof of Sah's lemma, 
we observe that
\[
\xi(\lambda g) = \xi(g \lambda) \Rightarrow (\lambda-1)\xi(g) = g \cdot \xi(\lambda) - \xi(\lambda).
\]
This formula shows both that $16\xi$ is a coboundary, and that $\xi(\lambda)$ is such that $g \cdot \xi(\lambda) - \xi(\lambda)$ is divisible by $16$ in $E[2^m]$. Imposing this condition for $g$ varying in a set of generators of $G_{2^\infty}$ (recall that we only have finitely many groups to test) we obtain that $\xi(\lambda)$ is divisible by 8. Let us write $\xi(\lambda)=8a$ for some (non-unique) $a \in E[2^m]$.
As a consequence, we have that for every $g \in G_{2^n}$
\[
8 \cdot 2 \xi(g) = g \cdot \xi(\lambda) - \xi(\lambda) = 8(g \cdot a - a).
\]
Letting $\psi$ be the coboundary $g \mapsto g\cdot a-a$ we then obtain that $2\xi$ is cohomologous to the cocycle $2\xi-\psi$, which by the above takes values in $E[8]$. A direct verification, for which we give details below, shows that $H^1(G_{2^n}, E[8])$ has exponent dividing $4$ for all $n \geq 3$. This implies in particular that $4 \cdot (2\xi) : G_{2^n} \to E[8]$ is a coboundary, hence a fortiori $8\xi : G_{2^n} \to E[2^m]$ is also a coboundary, and therefore $[8]$ kills $H^1(G_{2^n}, E[2^m])$ as desired.

To check that $H^1(G_{2^n}, E[8])$ has exponent dividing 4 we proceed as follows. Notice first that by Lemma \ref{lemma:CohomologyBoundReductionToGInfty} it suffices to show that $[4]$ is zero on $H^1(G_{2^\infty}, E[8])$. On the other hand, consider an element $g \in G_{2^\infty}$ that is the $8$-th power of an element $h$ congruent to the identity modulo $8$, and let $\xi : G_{2^\infty} \to E[8]$ be any cocycle. As $h$ acts trivially on $E[8]$, the restriction of $\xi$ to the subgroup generated by $h$ is a homomorphism, hence $\xi(g)=\xi(h^8)=8\xi(h)=0$. This proves that $\xi$ factors via the finite quotient
\[
G_{2^\infty} / \langle g^8 : g \equiv \Id \pmod{8} \rangle.
\]
For all the cases of interest we know from \cite{MR3500996} that $G_{2^\infty}$ contains all matrices congruent to 1 modulo $2^5$, hence $\langle g^8 : g \equiv \Id \pmod{8} \rangle$ contains all matrices congruent to $\Id$ modulo $2^8$. We are thus reduced to considering the group $Q := G_{2^8} / \langle g^8 : g \equiv \Id \pmod 8 \rangle$ and checking that the exponent of $H^1(Q, E[8])$ divides 4, which we do by explicit computations in MAGMA.
\end{proof}

In order to bound the exponent of $H^1(G_\infty,E[N])$  we will apply the following technical result,
which is worth stating in a general form.

\begin{proposition}
  \label{proposition:cohomologyBoundGeneric}
  Let $G_\infty$ be a subgroup of $\GL_2(\hat \Z)$ and for every
  prime $\ell$ denote by $G_{\ell^\infty}$ the projection of $G_\infty$ in
  $\GL_2(\Z_\ell)$. Let $J_\ell$ be the kernel of the projection
  $G_\infty\to \GL_2(\Z_\ell)$ and $\overline J_\ell$ be the image of $J_\ell$
  in $\prod_{p \text{ prime}}\GL_2(\F_p)$.
  Finally let $T$ be any $G_\infty$-submodule of $(\Q/\Z)^2$.
  Assume that for every prime $\ell$ there are a positive integer $a_\ell$ and
  non-negative integers $n_\ell,m_\ell$ such that the following hold:
  \begin{enumerate}
    \item For all but finitely many primes $\ell$ we have
          $v_\ell(a_\ell)=n_\ell=m_\ell=0$.
\item For every prime $\ell$ the exponent of $H^1(G_{\ell^\infty}, T[\ell^\infty])$ divides $\ell^{n_\ell}$.
    \item For every prime $\ell$ there is a scalar $g_\ell
          \in G_{\ell^\infty}$ such that
          $v_\ell(g_\ell-1)\leq m_\ell$.
    \item For every prime $\ell$ and every $x\in J_\ell$ the image of
          $[\tilde g_\ell,x^{a_\ell}]$ in $\overline J_\ell$ is
          contained in $[\overline J_\ell,\overline J_\ell]$ for some lift
          $\tilde g_\ell\in G_\infty$ of $g_\ell$, where $g_\ell$ is as above.
  \end{enumerate}
The cohomology group $H^1( G_\infty, T)$ has finite exponent dividing
    $\prod_{\ell}\ell^{n_\ell+m_\ell+v_\ell(a_\ell)}$.
\end{proposition}

\begin{proof}
  We will write elements $x$ of $G_\infty$ as sequences $(x_p)_p$ indexed by
  the prime numbers $p$, where each $x_p$ is in $\GL_2(\Z_p)$. Denoting the
  $\ell$-part of $T$ by $T[\ell^\infty]$ we have
  \begin{align*}
    \bigoplus_{\ell}T[\ell^\infty]
  \end{align*}
  and since cohomology of profinite groups  commutes with direct limits (see \cite[Proposition 4.18]{MR4174395}), hence with direct sums, we get
  \begin{align*}
    H^1(G_\infty, T)\cong \bigoplus_{\ell} H^1(G_\infty,T[\ell^\infty]).
  \end{align*}
  Fix now a prime $\ell$. The inflation-restriction exact sequence
  for $J_\ell \triangleleft G_\infty$ gives
  \begin{equation}\label{eq:InfRes}
    0\to H^1\left(G_{\ell^\infty},T[\ell^\infty]^{J_\ell}\right) \to
    H^1\left(G_\infty, T[\ell^\infty]\right)\to
    H^1(J_\ell,T[\ell^\infty])^{G_{\ell^\infty}}\,.
  \end{equation}
  Since $J_\ell$ acts trivially on $T[\ell^\infty]$ we have
  \begin{align*}
    T[\ell^\infty]^{J_\ell}=T[\ell^\infty]\quad \text{and}\quad
    H^1(J_\ell,T[\ell^\infty]) = \Hom(J_\ell, T[\ell^\infty]),
  \end{align*}
  and the action of $G_{\ell^\infty}$ on the latter group is given, for every
  $g\in G_{\ell^\infty}$, every $\varphi\in \Hom(J_\ell,T[\ell^\infty])$ and
  every $x\in J_\ell$, by
  \begin{align*}
    (g\varphi)(x)=g\varphi(\tilde g^{-1}x\tilde g)
  \end{align*}
  where $\tilde g\in G_\infty$ is any element mapping to $g$ in
  $G_{\ell^\infty}$ (see for example \cite[Theorem 4.1.20]{MR1282290}).  
By assumption, the cohomology group $H^1(G_{\ell^\infty},T[\ell^\infty]^{J_\ell})$ is killed by $\ell^{n_\ell}$. 

  Since every element of $T[\ell^\infty]$ has order a power of $\ell$ and
  the kernel of the quotient map $J_\ell\to\overline J_\ell$ is contained in
  the product of pro-$p$ groups for $p\neq \ell$, every group homomorphism from
  $J_\ell$ to $T[\ell^\infty]$ factors via $\overline{J}_\ell$.
  Moreover, since $T[\ell^\infty]$ is abelian, we have
  \begin{align*}
    \Hom(J_\ell,T[\ell^\infty])=\Hom(\overline J_\ell^{\ab},T[\ell^\infty])\,.
  \end{align*}

  Assume now that $\varphi \in \operatorname{Hom}(J_\ell, T[\ell^\infty])$ is
  $G_{\ell^\infty}$-invariant. For every $x\in J_\ell$ and any lift
  $\tilde g_\ell\in G_\infty$ of $g_\ell$ such that $[\tilde g_\ell,x^{a_\ell}]
  \in [\overline J_\ell, \overline J_\ell]$ (hence in particular $\varphi([\tilde g_\ell,x^{a_\ell}])=0$) we have
  \begin{align*}
    a_\ell\varphi(x)=\varphi(x^{a_\ell})=(g_\ell\varphi)(x^{a_\ell})=
    g_\ell\varphi(\tilde g_\ell^{-1}x^{a_\ell}\tilde g_\ell)=
    a_\ell g_\ell\varphi(x),
  \end{align*}
  so we get $a_\ell(g_\ell-1)\varphi(x)=0$. Since $v_\ell(g_\ell-1)\leq m_\ell$
  we have that $\Hom(J_\ell,T[\ell^\infty])^{G_{\ell^\infty}}$ is killed by
  $\ell^{m_\ell+v_\ell(a_\ell)}$. From these estimates and the exact sequence
  \eqref{eq:InfRes} we conclude that the exponent of $H^1(G_\infty,T)$ divides
  \begin{align*}
    \prod_{\ell}\ell^{n_\ell+m_\ell+v_\ell(a_\ell)}\,,
  \end{align*}
  as required.
\end{proof}

\begin{remark}
If, in the previous proposition, one does not assume that $g_\ell$ be a scalar, the conclusion still holds by letting $m_\ell$ be a non-negative integer such that $v_\ell(\det(g_\ell-\Id)) \leq m_\ell$. This may be established by a slight variation of the argument above: we only need to notice that $a_\ell(g_\ell-\Id)\varphi(x)=0$ implies $a_\ell \det(g_\ell-\Id) \varphi(x)=0$ (this can be seen for example by multiplying by the classical adjoint of $g_\ell-\Id$). The more specialised statement given above will allow us to obtain better numerical constants at the end.
\end{remark}

\begin{lemma} \label{lemma:SL2occurs}
  Let $G$ be a subgroup of $\GL_2(\hat\Z)$, let $\overline G$ be the image
  of $G$ under the quotient map $\GL_2(\hat\Z)\to\prod_{\ell \text{ prime}}\GL_2(\F_\ell)$, and
  let $p>5$ be a prime. If $\PSL_2(\F_p)$ occurs in $G$ (see §\ref{subsect:GroupTheory}), then $\overline G$
  contains $\SL_2(\F_p)\times\prod_{\ell \neq p}\set1$.
\end{lemma}
\begin{proof}
  Consider the kernel $N$ of the quotient map $G\to \prod_{\ell} \GL_2(\F_\ell)$. Every composition factor of $N$ is abelian, and a composition factor of $G$ that does not occur in $N$ must occur in $\overline G$. In particular, since $\PSL_2(\F_p)$ is simple and
  non-abelian, it must occur in $\overline G$. Consider now the projection
  $\overline G\to\prod_{\ell \neq p}\GL_2(\F_\ell)$ and let $N'$ be its kernel:
  since $\PSL_2(\F_p)$ does not occur in $\GL_2(\F_\ell)$ for $\ell\neq p$, it must
  occur in $N'$. Then by \cite[IV-25]{MR1484415} we must have that
  $\overline G$ contains $\SL_2(\F_p)\times\prod_{\ell\neq p}\set1$.
\end{proof}

We now come to our main result on the Galois cohomology of elliptic curves over $\Q$.

\begin{theorem} \label{theorem:boundCohomologyOverQ}
  Let $E$ be a non-CM elliptic curve over $\Q$ and let $N$ be a positive integer. The cohomology group
  \begin{align*}
    H^1(\Gal(\Q(E_{\tors})\mid \Q), E[N])
  \end{align*}
  has finite exponent dividing
  \begin{align*}
    e:=2^{12}\times 3^8\times 5^3 \times7^3\times 11^2\,.
  \end{align*}
\end{theorem}
\begin{proof}
  After fixing an isomorphism
  $E_{\tors}\cong ( \Q/\Z)^2$, let $G_\infty \subseteq \GL_2(\hat{\Z})$ be the
  image of the adelic Galois representation associated with $E/\Q$ and let
  $G_{\ell^\infty} ,J_\ell$ and $\overline J_\ell$ be as in the statement of
  Proposition \ref{proposition:cohomologyBoundGeneric}.
  For every prime $\ell$ we let $n_\ell$ be as in Equation \eqref{eq:n_ell}
  and $\lambda_\ell\in G_{\ell^\infty}$ be a scalar such that
  $v_\ell(\lambda_\ell-1)=n_\ell+v_\ell(2)$ and, for $\ell \geq 13$, such that
  $\lambda_\ell^2 \not \equiv 1 \pmod{\ell}$.
  The elements $\lambda_\ell$ exists by Theorem 
  \ref{thm:ScalarsImageGalois} and Proposition \ref{prop:SmallImprovement}.
    Let $g \in G_\infty$ be an element whose $\ell$-component is $\lambda_\ell$
  and set $\tilde g_\ell := g^2$. Finally, let
  \[
    a_\ell=\lcm\set{\exp\PGL_2(\F_p)\mid p\in\mathcal T_0,\,p\neq \ell}
  \]
  and $m_\ell=n_\ell+v_\ell(4)$. We now check that these choices satisfy all the assumptions of Proposition \ref{proposition:cohomologyBoundGeneric}, with $T=E[N]$. Clearly $v_\ell(a_\ell)=n_\ell=m_\ell=0$ for
  all but finitely many primes $\ell$, and
  one checks that $v_\ell(g_\ell-1)=m_\ell$ for all primes $\ell$. Theorem 
  \ref{theorem:ExpCohomologySmallPrimes} shows that $H^1(G_{\ell^\infty}, T[\ell^\infty])$ is killed by $\ell^{n_\ell}$. It only remains to check property (4), that is, we wish to prove that for every $x=(x_p)_p\in J_\ell$ the image $\overline h$ of
  $h=[\tilde{g}_\ell,x^{a_\ell}]$ in $\overline J_\ell$ is contained in
  $[\overline J_\ell,\overline J_\ell]$. To see this, notice first of all
  that the $\ell$-component of $\overline h$ in $\overline J_\ell$ is trivial,
  since $x_\ell=1$. The $p$-component of $\overline h$ is  trivial for
  every prime $p\in \mathcal T_0$, because $x_p^{a_\ell}\in\GL_2(\F_p)$ is a
  scalar (its image in $\PGL_2(\F_p)$ is trivial). Moreover, the $p$-component
  of $\overline{h}$ is also trivial for every prime $p\not\in\mathcal T_0$
  such that $G_p$ is contained in the normalizer of a Cartan subgroup. To see
  this, notice that $a_\ell$ is even and the $p$-component of $\tilde{g}_\ell$
  is a square (since $\tilde{g}_\ell$ itself is a square), so that both
  $(\tilde{g}_\ell)_p$ and $x^{a_\ell}$ belong to the Cartan subgroup itself,
  which is abelian.

  For all other primes $p$, the mod-$p$ Galois representation is surjective.
  Indeed by Theorem \ref{thm:Mazur} we know that $G_p$ acts irreducibly on
  $E[p]$ (since $p \not \in \mathcal{T}_0$), by \cite[p.~36]{MR488287} we know
  that $G_p$ is not contained in an exceptional subgroup, and by assumption
  $G_p$ is not contained in the normaliser of a Cartan subgroup. By Theorem
  \ref{thm:Dickson} we then obtain $\SL_2(\F_p) \subseteq G_p$, 
  so in particular $\PSL_2(\F_p)$ occurs in $G_\infty$. Since by
  \cite[p.~IV-25]{MR1484415} it cannot occur in $G_{\ell^\infty}$, which is a
  subgroup of $\GL_2(\Z_\ell)$, it must occur in $J_\ell$. Then by Lemma
  \ref{lemma:SL2occurs}, applied to $G = J_\ell$, we have that
  $S_p:=\SL_2(\F_p)\times \prod_{q\neq p}\set1$ is contained in
  $\overline J_\ell$ for such primes $p$.

  For each prime $p$, let $H_p$ be the trivial group if $p$ is in
  $\mathcal{T}_0$, if $\rho_{p}$ is not surjective, or if $p=\ell$, and
  let $H_p=\SL_2(\F_p)$ otherwise. By the above, we have
  $(\overline h)_p=\Id \in H_p$ for $p=\ell$, for all $p\in\mathcal{T}_0$, and
  for all $p$ such that $\rho_{p}$ is not surjective, and
  $(\overline h)_p \in H_p=\SL_2(\mathbb{F}_p)$ for all other $p$. We now
  show that $[\overline{J}_\ell, \overline{J}_\ell]$ contains $\prod_p H_p$.
  This product is topologically generated by the groups $S_p$ for
  $p\not\in\mathcal{T}_0 \cup\{\ell\}$ such that the mod-$p$ representation
  attached to $E$ is surjective, so it suffices to show that the closed
  subgroup $[\overline{J}_\ell, \overline{J}_\ell]$ contains $S_p$ for every
  such $p$. This follows from the fact that $\SL_2(\F_p)$ is a perfect group,
  that is it coincides with its own commutator subgroup, so
  $[\overline J_\ell,\overline J_\ell] \supseteq [S_p,S_p]=S_p$. 
  Thus we get
  $\overline{h} \in\prod_p H_p\subseteq[\overline{J}_\ell,\overline{J}_\ell]$.

  We have then checked all the hypotheses needed to apply Proposition
  \ref{proposition:cohomologyBoundGeneric}, and we conclude by noting that
  \begin{align*}
    v_\ell(a_\ell) = 
    \begin{cases}
      4 & \text{if }\ell=2,\\
      2 & \text{if }\ell=3,\\
      1 & \text{if }\ell=5,7,\\
      0 & \text{if }\ell\geq 11.
    \end{cases}
  \end{align*}
  
\end{proof}

In the CM case we can say something much stronger: we prove a bound that is valid for all number fields and only depends on the degree of the field of definition of the elliptic curve.
\begin{theorem}\label{thm:CMH1Bound}
Let $K$ be a number field of degree $d$ and let $E/K$ be an elliptic curve such that $E_{\overline{K}}$ has CM by an order $R$ in the quadratic imaginary field $F$. Let $h=\# R^\times \in \{2,4,6\}$ and $g = [FK:K] \in \{1,2\}$. For every prime $\ell$, let $e_\ell = \min_{a \in \mathbb{Z}_\ell^\times} v_\ell(a^{hd}-1)$. Then $e_\ell$ is finite for all primes $\ell$ and zero for all but finitely many $\ell$, and the exponent of the cohomology group $H^1(G_\infty, T)$ divides $g \prod_\ell \ell^{e_\ell}$ for all Galois submodules $T$ of $E_{\tors}$.
\end{theorem}
\begin{proof}
Let $H=\Gal(K_\infty \mid KF)$, so that $H$ is a subgroup of $G_\infty$ of index $g$ (recall that the field of complex multiplication is contained in $K_\infty$). Let $\operatorname{Cor}$ and $\operatorname{Res}$ denote respectively the corestriction map from $H^1(H, -)$ to $H^1(G_\infty, -)$ and the restriction map from $H^1(G_\infty, -)$ to $H^1(H, -)$. As is well-known, one has the equality $\mathrm{Cor}\circ \mathrm{Res} = [g]$. 
Let $e$ be the exponent of $H^1( G_\infty, T)$ and $e'$ be the exponent of $H^1(H, T)$. Observe now that $[e']$ is zero on $H^1( H, T)$, so one gets
\[
[ge'] = [e'] \circ \mathrm{Cor}\circ \mathrm{Res} = \mathrm{Cor} \circ [e'] \circ \mathrm{Res} = \mathrm{Cor} \circ [0] = [0]
\]
on $H^1( G_\infty, T )$. Thus the exponent of this latter group divides $ge'$; it now suffices to bound $e'$. 

By the theory of complex multiplication the Galois group $H$ is abelian. We identify this group to a subgroup of $\prod_\ell \GL_2(\Z_\ell)$, and regard $g \in H$ as a collection $(g_\ell)_{\ell}$ of elements in $\GL_2(\Z_\ell)$. Since $H$ is abelian, Lemma \ref{lemma:Sah} applies to any $(g_\ell)_{\ell} \in H$, so $H^1(H, T)$ is killed by $(g_\ell-1)_{\ell}$. Writing $H^1(H, T) = \bigoplus_\ell H^1(H, T[\ell^\infty])$, we see that each direct summand $H^1(H, T[\ell^\infty])$ (which is the pro-$\ell$ part of $H^1(H, T)$) is killed by $g_\ell-1$ for every $(g_\ell)_\ell \in H$.

Let now $H_{\ell^\infty}$ be the projection of $H$ to $\GL_2(\Z_\ell)$, or equivalently the image of the $\ell$-adic representation attached to $E/FK$.
We know from \cite[Theorem 6.6]{MR3766118} (or \cite[Theorem 1.1(a)]{MR4077686}) that $H_{\ell^\infty}$ is contained in $(R \otimes \Z_\ell)^\times$, and that $[ (R \otimes \Z_\ell)^\times : H_{\ell^\infty}] \mid \frac{h}{2} [FK:\Q]$. Notice that \cite[Theorem 6.6]{MR3766118} only gives an inequality, but it is clear from the proof that we actually have divisibility.
In particular, $[ \Z_\ell^\times : \Z_\ell^\times \cap H_{\ell^\infty} ]$ divides $\frac{h}{2} [FK:\Q]$, so for every $a \in \Z_\ell^\times$ the scalar $a^{h[FK:\Q]/2}$ is in $H_{\ell^\infty}$, and multiplication by $a^{h[FK:\Q]/2}-1$ kills $H^1(H, T[\ell^\infty])$. Notice that $h[FK:\Q]/2$ divides $hd$, so the same statement holds with $a^{h[FK:\Q]/2}-1$ replaced by $a^{hd}-1$.
As $H^1(H, T[\ell^\infty])$ is a (pro-)$\ell$ group, this shows that the exponent of $H^1(H, T[\ell^\infty])$ is finite and divides $\ell^{e_\ell}$. Finally, for $\ell-1 > hd$, choosing $a$ that is a primitive root modulo $\ell$ gives $v_\ell(a^{hd}-1)=0$, hence $e_\ell=0$ and $H^1(H, T[\ell^\infty])$ is trivial for all such primes. The theorem now follows from the fact that the exponent $e'$ of $H^1(H, T)$ is the least common multiple of the exponents of the groups $H^1(H, T[\ell^\infty])$ as $\ell$ varies among the primes.
\end{proof}

In the special case $K=\Q$ we may further improve the previous result.
\begin{proposition}\label{prop:CohomologyBoundCMQ}
Let $E/\Q$ be an elliptic curve such that $E_{\overline{\Q}}$ has CM. The exponent $e$ of the cohomology group $H^1(G_\infty, T)$ divides $2^2 \cdot 3$ for all Galois submodules $T$ of $E_{\tors}$.
\end{proposition}
\begin{proof}
Let $F$ be the field of complex multiplication of $E$, let $\mathcal{O}$ be the endomorphism ring of $E_F$, and let $H=\rho_{\infty}(\operatorname{Gal}(\overline{F}/F))$, considered as a subgroup of $\GL_2(\hat{\Z})$. There are inclusions $\hat{\Z}^\times \cap H \subseteq H \subseteq (\mathcal{O} \otimes \hat{\Z})^\times$, and $[\hat{\Z}^\times : \hat{\Z}^\times \cap H] \leq [(\mathcal{O} \otimes \hat{\Z})^\times : H]$. Suppose first $j \not \in \{0,1728\}$. Then $[(\mathcal{O} \otimes \hat{\Z})^\times : H] \leq 2$ by \cite[Corollary 1.5]{MR4077686}, hence $[\hat{\Z}^\times : \hat{\Z}^\times \cap H] \leq 2$. This implies easily that $H$ (hence $G_\infty$) contains an element $\lambda = (\lambda_\ell) \in \prod_\ell \Z_\ell^\times = \hat{\Z}^\times$ with $v_2(\lambda_2-1) \leq 2$, $v_3(\lambda_3-1) \leq 1$ and $v_\ell(\lambda_\ell-1) =0 $ for all $\ell \geq 5$ (for $\ell=2$ notice that a subgroup of index at most $2$ of $\hat{\Z}^\times$ cannot be trivial modulo 8). The claim in this case thus follows from Lemma \ref{lemma:Sah}. When $j \in \{0,1728\}$ the argument is similar, but one also needs to rely on the classification of the possible $\ell$-adic images of Galois for $\ell \leq 7$ provided by \cite{2018arXiv180902584L}. We give some more details for $\ell=2$, the other cases being similar and easier.

Suppose that all the scalars $\lambda=(\lambda_\ell)$ in $H \cap \hat{\Z}^\times$ satisfy $v_2(\lambda_2-1) \geq 3$. Then $[\hat{\Z}^\times : \hat{\Z}^\times \cap H]$ is a multiple of $4$, which (since $\hat{\Z}^\times$ is a normal subgroup of $H$) implies $4 \mid [(\mathcal{O} \otimes \hat{\Z})^\times : H]$. Due to \cite[Corollary 1.5]{MR4077686} this must be an equality, and we must have $\mathcal{O}=\mathbb{Z}[i]$ and $j=1728$. On the other hand, from the proof of Theorem \ref{thm:CMH1Bound} we know that the $2$-part of the exponent of $H^1(G_\infty, T)$ is at most twice the $2$-part of the exponent of $H^1(H, T)$, so if the latter is not divisible by 4 we are already done. Moreover, $4$ can divide this exponent only if all the scalars in $\rho_{2^\infty}(\operatorname{Gal}(\overline{F}/F))$ are congruent to 1 modulo 4. By \cite[Theorem 1.7]{2018arXiv180902584L}, this implies that $[ (\mathcal{O} \otimes \Z_2)^\times : \rho_{2^\infty}(\operatorname{Gal}(\overline{F}/F)) ] = 4$. Combined with $[(\mathcal{O} \otimes \hat{\Z})^\times : H]=4$, this shows that $H$ is the product $\rho_{2^\infty}(\Gal(\overline{F}/F)) \times \prod_{\ell \geq 3} (\mathcal{O} \otimes \Z_\ell)^\times$. By \cite[Theorem 1.7]{2018arXiv180902584L} again, the factor $\rho_{2^\infty}(\Gal(\overline{F}/F))$ contains a scalar $\lambda_2$ with $v_2(\lambda_2-1)=2$. Since $H$ is the above direct product, we obtain that $H$ (hence $G_\infty$) contains $(\lambda_2, -1, -1, \ldots)$. Applying Sah's lemma to this element then shows that the 2-part of the exponent of $H^1(G_\infty, T)$ divides 4.
\end{proof}

To conclude this section we discuss the case of \textit{Serre curves}, namely those elliptic curves over $\mathbb{Q}$ for which $[\GL_2(\hat{\Z}) : G_\infty]$ is minimal (hence equal to $2$, see \cite{Serre}). It is known that, when ordered by height, 100\% of elliptic curves over $\Q$ are Serre curves \cite{MR2563740}, so our next theorem describes the `generic' situation. The proof combines many of the same ingredients that already appear in Theorems \ref{theorem:boundCohomologyOverQ} and \ref{theorem:ExpCohomologySmallPrimes}.

\begin{theorem}
Suppose $E/\Q$ is a Serre curve. For every Galois submodule $T$ of $E_{\tors}$ we have
\[
H^1(G_\infty, T) = \begin{cases}
\Z/2\Z, \text{ if } T[2] \neq \{0\} \\
\{0\}, \text{ if } T[2] = \{0\}.
\end{cases}
\]
\end{theorem}
\begin{proof}
The description of Serre curves given in \cite[Section 5]{MR2563740} implies that $G_\infty$ contains $\SL_2(\hat{\Z})$. We will make use of two special elements of $\SL_2(\hat{\Z}) \subset G_\infty$: one is $-\Id$, while the other is $h=(h_2,\Id,\Id,\ldots)$, where $h_2=\begin{pmatrix}
0 & -1 \\
1 & -1
\end{pmatrix} \in \SL_2(\Z_2)$. Notice that $h_2-\Id$ is invertible over $\Z_2$.
Let $\xi : G_\infty \to E_{\tors}$ be any cocycle and let $g \in G_\infty$ be arbitrary. We have the equality
\[
\xi(-\Id) - \xi(g) =  \xi((-\Id) \cdot g) = \xi(g \cdot (-\Id)) = \xi(g) + g \xi(-\Id).
\]
Choosing $g=h$ gives $-2\xi(h) = (h-\Id) \cdot \xi(-\Id)$ in $T = \bigoplus_\ell T[\ell^\infty]$. Taking into account that the $2$-adic component of $h-\Id$ is invertible, while multiplication by $2$ is invertible on $T[\ell^\infty]$ for each $\ell > 2$, we obtain that $\xi(-\Id)$ is divisible by $2$ in $T$. Writing $\xi(-\Id)=-2a$ for some $a \in T$ we then have $2(\xi(g) - (g \cdot a-a))=0$, that is, the cocycle $\xi$ is cohomologous to the cocycle $g \mapsto \xi(g) - (g\cdot a- a)$ with values in $T[2]$. 

We have thus shown that the natural map $H^1(G_\infty, T[2]) \to H^1(G_\infty, T)$ is surjective. It is also injective, as one sees by taking the cohomology of the exact sequence $0 \to T[2] \to T \to 2T \to 0$ and observing that $H^0(G_\infty, T)=H^0(G_\infty, 2T)=(0)$. 
Hence $H^1(G_\infty, T) = H^1(G_\infty, T[2])$. We now describe this group. 
Let $N=\ker(G_\infty \to G_{2^\infty})$, so that $G_\infty/N \cong G_{2^\infty}=\GL_2(\Z_2)$. The inflation-restriction sequence yields
\[
0 \to H^1(G/N, T[2]) \to H^1(G_\infty,T[2]) \to H^1(N, T[2])^{G_\infty},
\]
so it suffices to show that $H^1(\GL_2(\Z_2), T[2])$ is either trivial or isomorphic to $\mathbb{Z}/2\mathbb{Z}$ according to whether $T[2]$ is trivial or not, while $H^1(N, T[2])^{G_\infty}$ vanishes. We prove the latter statement first. Since $N$ acts trivially on $T[2]$ by construction we have $H^1(N, T[2])^{G_\infty} = \operatorname{Hom}(N, T[2])^{G_\infty}$. 
The conjugation action of $h \in G_\infty$ on $N$ is trivial (the only nontrivial coordinate of $h$ is $h_2$, while elements of $N$ have trivial $2$-adic component), so a homomorphism $\varphi \in \operatorname{Hom}(N, T[2])$ is $h$-invariant if and only if for all $n \in N$ we have $\varphi(n)=(h \varphi)(n)=h \cdot \varphi(h^{-1}nh)=h \cdot \varphi(n)$. Since $h$ acts on $T[2]$ via $h_2$, which has no nonzero fixed points on $T[2]$, this implies that the only $h$-invariant homomorphism $N \to T[2]$ is the trivial one. Thus $H^1(N, T[2])^{G_\infty}$ vanishes as claimed. Finally consider $H^1(\GL_2(\Z_2), T[2])$.  
Notice that $T[2]$ is a Galois submodule of $E[2]$, so we either have $T[2]=E[2]$ or $T[2]=\{0\}$. In the latter case the cohomology group certainly vanishes, so we can assume $T[2]=E[2]$.
As in the proof of Theorem \ref{theorem:ExpCohomologySmallPrimes}, every cocycle $\GL_2(\Z_2) \to E[2]$ factors via $\GL_2(\Z_2) / \langle g^2 : g \equiv \Id \pmod{2} \rangle$, hence in particular via $\GL_2(\Z/8\Z)$. Thus it suffices to check that $H^1(\GL_2(\Z/8\Z), E[2]) = \mathbb{Z}/2\mathbb{Z}$, which is easy to do directly.
\end{proof}

%% file: 5-AlgebraA.tex
\section{The algebra $\mathbb{Z}_\ell[G_{\ell^\infty}]$}\label{sect:GaloisImageAlgebra}
Following the strategy suggested by \cite[Proposition 4.12]{2019arXiv190905376L}, in order to study the degrees of Kummer extensions in the next section we now study the algebra $A=\mathbb{Z}_\ell[G_{\ell^\infty}]$, by which we mean the \textit{closed} subalgebra of $\Mat_{2 \times 2}(\Z_\ell)$ generated by $G_{\ell^\infty} \subseteq \Mat_{2 \times 2}(\Z_\ell)$.
The hardest case is when the action of $G_{\ell}$ on $E[\ell]$ is reducible, and to handle this situation we rely on the following general estimate for $A$.

\begin{proposition} \label{proposition:algebraReducibleGeneral}
Let $E$ be an elliptic curve over a number field $K$ having at least one real place. Let $\ell > 2$ be a prime number. Suppose that $G_\ell$ acts reducibly on $E[\ell]$ and let $\ell^m$ be the maximal degree of an $\ell$-power cyclic isogeny $E \to E'$ defined over $K$. The algebra $A=\mathbb{Z}_\ell[G_{\ell^\infty}]$ contains
$\ell^{m}\Mat_{2\times 2}(\Z_\ell)$.
\end{proposition}
\begin{proof}
We claim that there exists a basis of $T_\ell E$ with respect to which $G_{\ell^\infty}$ contains $\matr 100{-1}$. To see this, let $\tau \in \Gal(\overline{K}/K)$ be a complex conjugation, corresponding to a real embedding $K \hookrightarrow \mathbb{R}$ (one exists by assumption), and let $h = \rho_{\ell^\infty}(\tau)$. Then we have $h^2=\Id$ and $\det h = \chi_{\ell^\infty}(\tau)=-1$, which implies that the eigenvalues of $h$ are $\pm 1$. It follows that $h$ can be diagonalised over $\mathbb{Q}_\ell$, and also over $\mathbb{Z}_\ell$ since its eigenvalues are distinct modulo $\ell \neq 2$. As the conclusion of the proposition is independent of the choice of basis, we may assume that $h = \rho_{\ell^\infty}(\tau)=\matr 100{-1} \in A$. It follows that $E_{11}:= \matr 1000 = \frac{1}{2}(1+h)$ and $E_{22} := \matr 0001 = \frac{1}{2}(1-h)$ are in $A$. 
By assumption, $E$ does not admit 
a cyclic isogeny of degree $\ell^{m+1}$ defined over $K$. In terms of the matrix representation of the Galois action, this implies in particular that $G_{\ell^{\infty}}$ contains a matrix $M_1$ whose coefficient in position $(2,1)$ is nonzero modulo $\ell^{m+1}$ (for otherwise, $\langle \begin{pmatrix}
1 \\ 0
\end{pmatrix} \rangle \subset (\mathbb{Z}/\ell^{m+1} \mathbb{Z})^2 \cong E[\ell^{m+1}]$ would be a Galois-stable cyclic subgroup of order $\ell^{m+1}$), and similarly it also contains a matrix $M_2$ whose $(1,2)$-coefficient is nonzero modulo $\ell^{m_\ell+1}$. Thus we have $E_{22}M_1E_{11} = \begin{pmatrix}
0 & 0 \\
a & 0
\end{pmatrix}$ with $v_\ell(a) \leq m$ and $E_{11}M_2E_{22} = \begin{pmatrix}
0 & b \\
0 & 0
\end{pmatrix}$ with $v_\ell(b) \leq m$. The four matrices $E_{11}, E_{22}, E_{22}M_1E_{11}$ and $E_{11}M_2E_{22}$ are all in $A$, and their $\mathbb{Z}_\ell$-span contains $\ell^{m} \operatorname{Mat}_{2 \times 2}(\mathbb{Z}_\ell)$.
\end{proof}

\begin{remark}
The exponent $m$ is optimal. Indeed, if $E$ admits a $K$-rational isogeny of degree $\ell^m$, choosing a suitable basis of $T_\ell E$ we can ensure that $G_{\ell^m}$ consists of upper-triangular matrices. In particular, the (2,1)-coefficient of all matrices in $\Z_\ell[G_{\ell^\infty}]$ is divisible by $\ell^{m}$, so that the result cannot be improved.
\end{remark}

We also give a variant of the previous result for $\ell=2$. Notice that in this case we do not require that $E[2]$ be reducible.

\begin{proposition}\label{prop:Algebra2General}
Let $E$ be an elliptic curve over a number field $K$ having at least one real place. Let $2^m$ be the maximal degree of a $2$-power cyclic isogeny $E \to E'$ defined over $K$ (including $m=0$ if there are no such isogenies). The algebra $A=\mathbb{Z}_2[G_{2^\infty}]$ contains
$2^{m+1}\Mat_{2\times 2}(\Z_2)$.
\end{proposition}
\begin{proof}
Let $\tau \in \Gal(\overline{K} \mid K)$ be a complex conjugation. 
There is a basis of $T_2E$ whose first element is fixed by $\rho_{2^\infty}(\tau)$:
indeed, $\tau$ fixes all torsion points in $E(\mathbb{R})$, whose identity component is isomorphic to the circle group, hence contains a compatible family of $2^n$-torsion points.
It follows easily that $\rho_{2^\infty}(\tau)$ 
is $\GL_2(\Z_2)$-conjugate to either $\begin{pmatrix}
1 & 0 \\ 0 & -1
\end{pmatrix}$ or $\begin{pmatrix}
1 & 1 \\ 0 & -1
\end{pmatrix}$. In the first case one may reason as in Proposition \ref{proposition:algebraReducibleGeneral} to obtain that $\Z_2[G_{2^\infty}]$ contains $2E_{11}, 2E_{22}, 2E_{22}M_1$, and $2M_2E_{22}$, hence that it contains $2^{m+1}\Mat_{2 \times 2}(\Z_2)$. In the second case, suppose first that $G_2$ acts on $E[2]$ with a fixed point $P$, which is necessarily the first 2-torsion point in the given basis of $E[2] \cong T_2E/2T_2E$. 
Let $E \to E'$ be the $2$-isogeny with kernel $\langle P \rangle$. 
The $2$-adic representations attached to $E, E'$ differ by conjugation by $\matr 2001$. The $2$-adic representation attached to $E'$ maps $\tau$ to $\matr 120{-1}$, which is $\GL_2(\Z_2)$-conjugate to $\matr 100{-1}$. Moreover, the maximal degree of a $2$-power isogeny $E' \to E''$ is at most $2^{\max\{m-1,1\}}$
The previous arguments then apply to $E'$, 
hence the corresponding algebra $A'$ contains $2^{\max\{m-1,1\}+1}\Mat_{2\times 2}(\Z_2)$. Conjugating back we find that $A$ contains $2^{\max\{m,2\}+1}\Mat_{2\times 2}(\Z_2)$, and a direct check for $m=1$ finishes the proof in this case.
Finally, if $\rho_{2^\infty}(\tau)=\begin{pmatrix}
1 & 1 \\ 0 & -1
\end{pmatrix}$ and $E[2]$ is an irreducible Galois module (hence $m=0$), then $G_2=\GL_2(\F_2)$ (notice that $\#G_2$ is even since $\rho_2(\tau)$ is nontrivial). This implies $G_2=\GL_2(\F_2)$, from which it follows that the reduction modulo $2$ of $A$ is all of $\Mat_{2 \times 2}(\F_\ell)$. By Nakayama's lemma we obtain $A=\Mat_{2 \times 2}(\Z_2)$.
\end{proof}

For the irreducible case (and $\ell >2$) we rely instead on the following two observations.
The first one is well-known (see for example
\cite[Remark after Theorem 2]{MR1094193}); it is usually stated for elliptic curves over $\Q$, but -- as in the previous propositions -- it only depends on the number field having a real place.
\begin{lemma}\label{lemma:AbsIrred}
Let $K$ be a number field having at least one real place, $\ell>2$ be a prime number, $E/K$ be an elliptic curve, and $G_\ell \subseteq \GL_2(\F_\ell)$ be the image of the mod-$\ell$ Galois representation. The action of $G_\ell$ on $E[\ell]$ is either reducible or absolutely irreducible.
\end{lemma}

\begin{corollary} \label{corollary:algebraIrreducible}
Let $K$ be a number field having at least one real place, $\ell>2$ be a prime number, and $E/K$ be an elliptic curve. If $E[\ell]$ is an irreducible Galois module, then the algebra $A=\Z_\ell[G_{\ell^\infty}]$ is all of $\Mat_{2 \times 2}(\Z_\ell)$.
\end{corollary}
\begin{proof}
Let $\overline{A} \subseteq \Mat_{2 \times 2}(\mathbb{F}_\ell)$ be the image of $A$ under reduction modulo $\ell$.
By Nakayama's lemma, it suffices to prove that $\overline{A} = \Mat_{2 \times 2}(\F_\ell)$. Notice that $\overline{A}=\F_\ell[G_{\ell}]$.
As $G_\ell$ acts irreducibly on $E[\ell] \cong \F_\ell^2$ by assumption, Lemma \ref{lemma:AbsIrred} shows that it also acts irreducibly on $E[\ell] \otimes_{\F_\ell} \overline{\F_\ell}$, hence the natural module $\overline{\F_\ell}^2$ for $\overline{A} \otimes_{\F_\ell} \overline{\F_\ell}$ is irreducible. By \cite[Theorem 3.2.2]{etingof} we obtain $\overline{A} \otimes_{\F_\ell} \overline{\F_\ell} = \Mat_{2 \times 2}(\overline{\F_\ell})$, which implies $\overline{A}= \Mat_{2 \times 2}(\F_\ell)$.
\end{proof}

We now specialise to the case $K=\Q$.
For $\ell=2$ we have the following.
\begin{proposition} \label{proposition:algebra2}
Let $E$ be an elliptic curve defined over $\Q$. The algebra $\Z_2[G_{2^\infty}]$ contains $2^4\Mat_{2\times 2}(\Z_2)$, and if $E$ has potential complex multiplication it also contains $2^3\Mat_{2\times 2}(\Z_2)$.
\end{proposition}
\begin{proof}
If $E$ does not have complex multiplication over $\overline\Q$ we can check
  the claim directly by a short computer calculation, looping over all
  subgroups of $\GL_2(\Z_2)$ that can arise as the image of the 2-adic
  representation (the list of such groups is known as a consequence of the
  results in \cite{MR3500996}).
  If $E$ has CM over $\overline \Q$, then every $2$-power isogeny $E \to E'$ defined over $\Q$ has degree dividing 4 (see for example \cite[Remark 5.2]{MR4171427}). It follows from Proposition \ref{prop:Algebra2General} that $A$ contains $2^3\Mat_{2\times 2}(\Z_2)$.
\end{proof}
\begin{remark}
The result is optimal. This follows from \cite{MR3500996} in the non-CM case, while in the CM case it suffices to consider an elliptic curve with CM by $\mathbb{Z}[\sqrt{-4}]$, see \cite[Theorem 1.6]{2018arXiv180902584L}.
\end{remark}

We are now ready to obtain a uniform lower bound on the algebra $A$.
\begin{theorem} \label{theorem:algebraZG}
  Let $E$ be an elliptic curve over $\mathbb{Q}$ and let $\ell$ be a prime
  number. Set
\[
    m_{\text{non-CM}, \ell} = \begin{cases}
      4, \text{ if } \ell=2 \\
      2, \text{ if } \ell=3, 5 \\
      1, \text{ if } \ell=7,11,13,17,37\\
      0, \text{ otherwise}
    \end{cases}
    \quad 
        m_{\text{CM}, \ell} = \begin{cases}
      3, \text{ if } \ell=2,3 \\
      1, \text{ if } \ell=7,11,19,43,67,163 \\
      0, \text{ otherwise}
    \end{cases}
\]
  and $m_\ell=m_{\text{CM}, \ell}$ or $m_\ell=m_{\text{non-CM}, \ell}$ according to whether or not $E_{\overline{\Q}}$ has CM.
  The algebra $A=\mathbb{Z}_\ell[G_{\ell^\infty}]$ contains
  $\ell^{m_\ell}\Mat_{2\times 2}(\Z_\ell)$.
\end{theorem}
\begin{proof}
  The case $\ell=2$ is covered by Proposition \ref{proposition:algebra2}.
If $\ell\not\in \mathcal T_0\cup\set{19,43,67,163}$ (or just $\ell\not\in \mathcal T_0$ if $E$ is not CM), by
  Theorem \ref{thm:Mazur} the curve $E$ does not admit any rational subgroup of
  order $\ell$, so $E[\ell]$ is irreducible as a $G_\ell$-module and we can
  apply Corollary \ref{corollary:algebraIrreducible}.
For the remaining cases we apply Proposition
  \ref{proposition:algebraReducibleGeneral}, reading from \cite[Theorem 1]{MR675184} the maximal degrees of cyclic isogenies of $\ell$-power degree. Notice that isogenies of degree $3^3$ are possible only for CM elliptic curves, see \cite[p.~229]{MR0337974}. Also notice that $\ell$-isogenies between rational CM elliptic curves are only possible for $\ell \in \{2,3,7,11,19,43,167\}$, as follows for example from \cite[§5]{MR4171427}.
\end{proof}

%% file: 6-KummerBound.tex

\section{Kummer degrees}\label{sect:KummerDegrees}

Let $E$ be an elliptic curve over a number field $K$ and let
$\alpha\in E(K)$ be a point of infinite order. We give a brief description of the construction of the Kummer extensions of $K$ attached to $(E,\alpha)$, and refer the reader to \cite[§2.3]{2019arXiv190905376L}, \cite[Section 3]{JonesRouse},  \cite{2018arXiv180208527B}, or \cite{lombardo_perucca_2019} for more details.

Let $(M, N)$ be either a pair of positive integers with $N \mid M$, or $(\infty,N)$ with $N$ a positive integer.
We define $K_{M, N}$ as the extension of $K_M$ generated by the coordinates of all points $\beta \in E(\overline{K})$ such that $N\beta=\alpha$. The homomorphism 
\begin{equation}\label{eq:KummerMap}
\begin{array}{cccc}
\kappa_{M,N} : & \Gal(\overline{K} \mid K_M) & \to & E[N] \\
& \sigma & \mapsto & \sigma(\beta)-\beta
\end{array}
\end{equation}
is independent of the choice of $\beta \in E(\overline{K})$ such that $N\beta=\alpha$,
and has kernel $\Gal(\overline{K} \mid K_{M,N})$, hence identifies $\Gal(K_{M,N} \mid K_M)$ with a subgroup of $E[N]$. We will also need to pass to the limit in $N$: if $\ell$ is a prime number, we denote by $K_{\infty, 	\ell^\infty}$ the extension of $K_\infty$ generated by the coordinates of the points $\beta \in E(\overline{K})$ that satisfy $\ell^n\beta = \alpha$ for some $n \geq 0$. Similarly, we write $K_{\infty, \infty}$ for the extension of $K_\infty$ generated by the coordinates of the points $\beta \in E(\overline{K})$ that satisfy $N\beta = \alpha$ for some $N \geq 1$. Passing to the limit in $N$ in Equation \eqref{eq:KummerMap} we obtain an identification of $\Gal(K_{\infty, \ell^\infty} \mid K_\infty)$ with a $\Z_\ell$-submodule $V_{\ell^\infty}$ of $T_\ell E \cong \Z_\ell^2$, and of $\Gal(K_{\infty, \infty} \mid K_\infty)$ with a $\hat{\Z}$-submodule $V_\infty$ of $TE \cong \hat{\Z}^2$. 
We remark that $V_{\ell^\infty}$ is the projection of $V_{\infty}$ to $\Z_\ell^2$,
and since $V_{\ell^\infty}$ is a pro-$\ell$ group and there are no nontrivial
continuous morphisms from a pro-$\ell$ group to a pro-$\ell'$ group for
$\ell \neq \ell'$ we have $V_\infty = \prod_\ell V_{\ell^\infty}$. Finally, we recall the following fact, which will be crucial in our applications.

\begin{lemma}[{\cite[Lemma 2.5]{2019arXiv190905376L}}]\label{lemma:HnActionOnVn} For every prime $\ell$, the $\Z_\ell$-module $V_{\ell^\infty} \subseteq \Z_\ell^2$ is also a module for the natural action of $G_{\ell^\infty} \subseteq \GL_2(\Z_\ell)$ on $\Z_\ell^2$.
\end{lemma}

We are interested in studying the degrees
\begin{align}
  \left[K_{M,N}:K_M\right]
\end{align}
as the positive integers $N\mid M$ vary. As explained above, the Galois group $\Gal(K_{M,N} \mid K_M)$ is isomorphic to a subgroup of $E[N]$, which has order $N^2$, so the ratio
\begin{align}
  \label{eqn:kummerFailure}
  \frac{N^2}{\left[K_{M,N}:K_M\right]}
\end{align}
is an integer. It is well-known that \eqref{eqn:kummerFailure} is
bounded independently of the integers $M$ and $N$
(see for example \cite[Théorème 1]{Durham}, \cite[Lemme 14]{Hindry}, or \cite{MR552524}). In \cite{2019arXiv190905376L} we
have shown that, if $K=\Q$ and the image of $\alpha$ in the free abelian group $E(K)/E(K)_{\tors}$ is not divisible by any $n>1$, this ratio can be bounded independently also of $E$ and $\alpha$. We will now provide an explicit value for this bound.

\begin{remark} \label{remark:degreesOverKinfty}
  It is immediate to check that the ratio \eqref{eqn:kummerFailure} divides
$\displaystyle 
    \frac{N^2}{\left[K_{\infty,N}:K_\infty\right]}
$‚
  which in turn divides the index of $V_\infty$ in $\hat\Z^2$.
\end{remark}

\begin{lemma} \label{lemma:781}
  Let $E$ be an elliptic curve over a number field $K$ and let $\alpha\in E(K)$
  be a point whose image in the free abelian group
  $E(K)/E(K)_{\tors}$ is not divisible by any $n>1$. 
Let $e$ be a positive integer such that, for all positive integers $N$, the group $H^1(G_\infty,E[N])$ has exponent dividing $e$. 
 For every prime $\ell$ the group $V_{\ell^\infty}$ contains
  an element of $\ell$-adic valuation at most $v_\ell(e)$.
\end{lemma}
\begin{proof}
  This follows immediately from \cite[Lemma 7.8(1)]{2019arXiv190905376L} since for any positive integers $M,N$ with $N\mid M$ the exponent
  of $H^1(G_M,E[N])$ divides $e$ (Lemma \ref{lemma:CohomologyBoundReductionToGInfty}).
\end{proof}

\begin{lemma} \label{lemma:78234-bis}
  Let $E$ be an elliptic curve over a number field $K$ and let $\alpha\in E(K)$.
Suppose that $V_{\ell^\infty}$ contains an element $v$
  of $\ell$-adic valuation at most $d$ and that $\Z_\ell[G_{\ell^\infty}]
  \supseteq\ell^n\Mat_{2\times 2}(\Z_\ell)$ for some non-negative integer $n$.
  Then $[T_\ell E :V_{\ell^\infty}]$ divides $\ell^{n+2d}$.
\end{lemma}
\begin{proof} We may assume without loss of generality that $v$ has exact valuation $d$.
Up to a choice of isomorphism $T_\ell E \cong \mathbb{Z}_\ell^2$ we may then further assume $v=\ell^d \begin{pmatrix}
1 \\ 0
\end{pmatrix}$. The $\Z_\ell[G_{\ell^\infty}]$-module $V_{\ell^\infty}$ contains $\ell^n \Mat_{2 \times 2}(\Z_\ell)\cdot v$, hence in particular contains $\ell^{n+d} \begin{pmatrix}
0 \\ 1
\end{pmatrix}$, and the claim follows immediately.
\end{proof}

\begin{theorem} \label{theorem:mainTheoremKummer}
  Let $E$ be an elliptic curve defined over $\Q$ and let
  \[
  B_{\text{non-CM}} := (2^{24}\times 3^{16}\times 5^6 \times7^6\times 11^4) \times (2^4 \times 3^2 \times 5^2 \times 7 \times 11 \times 13 \times 17 \times 37)
  \]
  \[
  B_{\text{CM}} := (2^4  \times 3^2) \times (2^3 \times 3^3 \times 7 \times 11 \times 19 \times 43 \times 67 \times 163).
  \]
Set $B=B_{\text{CM}}$ or $B=B_{\text{non-CM}}$ according to whether or not $E_{\overline{\Q}}$ has complex multiplication.
  For all positive integers
  $M$ and $N$ with $N\mid M$ the ratio \eqref{eqn:kummerFailure} divides $B$.
\end{theorem}
\begin{proof}
  Let $e$ be a positive integer such that $[e]$ kills $H^1(G_\infty, E[N])$ for all positive integers $N$.
For every
  prime $\ell$ let $m_\ell$ be a non-negative integer such that $\Z_\ell[G_{\ell^\infty}]$ contains $\ell^{m_\ell}\Mat_{2 \times 2}(\Z_\ell)$.
  As explained above, the ratio \eqref{eqn:kummerFailure} divides  
  \begin{align*}
    [\hat\Z^2:V_{\infty}]=\prod_\ell[\Z_\ell^2:V_{\ell^\infty}] \,,
  \end{align*}
  and by Lemmas \ref{lemma:781} and \ref{lemma:78234-bis} we have that
  \begin{align*}
    [\Z_\ell^2:V_{\ell^\infty}] \qquad \text{divides} \qquad
    \ell^{m_\ell+2v_\ell(e)}.
  \end{align*}
The conclusion then follows by taking $e$ as in Theorem \ref{theorem:boundCohomologyOverQ} (for the non-CM case) or as in Proposition \ref{prop:CohomologyBoundCMQ} (for the CM case), and $m_\ell$ as in Theorem \ref{theorem:algebraZG}.
\end{proof}

\begin{remark}
	Taking into account Remark \ref{rem:RouseExponent3}, one can take
	$v_3(e)=6$ instead of $8$ in Theorem \ref{theorem:boundCohomologyOverQ},
	so that the exponent of $3$ in $B_{\text{non-CM}}$ can be improved from
	$18$ to $14$.
\end{remark}

%% file: 8-Examples.tex

\section{Examples}\label{sect:Examples}
In this short section we give examples showing that most of our results are sharp or close to being sharp. We start with Theorems \ref{theorem:ExpCohomologySmallPrimes} and \ref{theorem:boundCohomologyOverQ}.
For every positive integer $N$ we have an exact sequence of $\Gal(\overline{\Q} \mid \Q)$-modules
\[
0 \to E[N] \to E_{\tors} \xrightarrow{[N]} E_{\tors} \to 0,
\]
and taking Galois cohomology we get
\[
0 \to \frac{E(\Q)_{\tors}}{NE(\Q)_{\tors}} \to H^1(G_\infty, E[N]) \to H^1(G_\infty, E_{\tors})[N] \to 0.
\]
As it is well-known that there exist elliptic curves over $\Q$ with torsion points of order $2^3, 3^2, 5, 7$, taking $N$ equal to each of these numbers in turn shows that the constant of Theorem \ref{theorem:boundCohomologyOverQ} has to be divisible at least by $2^3 \cdot 3^2 \cdot 5 \cdot 7$. Moreover, by \cite[Theorem 1]{LawsonWutrich} we know that there exists an elliptic curve $E/\Q$ with $H^1( G_{11}, E[11]) \neq 0$. Thus in particular all the primes appearing in the constant of Theorem \ref{theorem:boundCohomologyOverQ} are necessary. A simple variant of this argument, working with $E[\ell^\infty]$ instead of $E_{\operatorname{tors}}$, also shows that Theorem \ref{theorem:ExpCohomologySmallPrimes} is optimal at least for $\ell \neq 3$. As already remarked in the introduction we do not seek to obtain the best possible value for $\ell=3$, but in any case our estimate is not far from sharp: the previous argument shows that the optimal value of $n_3$ is at least 2, while Theorem \ref{theorem:ExpCohomologySmallPrimes} shows that 3 suffices.

Consider now the CM case and Proposition \ref{prop:CohomologyBoundCMQ}. The elliptic curve with LMFDB label 27.a2 \cite[\href{https://www.lmfdb.org/EllipticCurve/Q/27/a/2}{27.a2}]{lmfdb} admits a rational 3-torsion point and no other 3-isogenies defined over $\Q$, hence it satisfies the hypotheses of \cite[Theorem 1]{LawsonWutrich}, which proves that for this curve $H^1(G_3, E[3]) \neq 0$. Thus the factor 3 in Proposition \ref{prop:CohomologyBoundCMQ} is necessary. As for the power of 2, the curve with LMFDB label 32.a2 \cite[\href{https://www.lmfdb.org/EllipticCurve/Q/32/a/2}{32.a2}]{lmfdb} has potential CM and a rational 4-torsion point, which as above shows that $H^1(G_{2^\infty}, E[4])$ has exponent 4. Thus Proposition \ref{prop:CohomologyBoundCMQ} is sharp.

Finally we turn to the primes that can appear in the ratio of Equation \eqref{eqn:kummerFailure}.
In order to find examples where a given prime $\ell$ divides the degree \eqref{eqn:kummerFailure} we proceed as follows. Let $E/\Q$ be a rational elliptic curve and let $P \in E(\Q)$ be a point not divisible by any $n>1$ in $E(\Q)/E(\Q)_{\tors}$. For a fixed prime $\ell>2$, we write the multiplication by $\ell$ map as
\[
[\ell](x,y) = \left( \frac{\phi_\ell(x)}{\psi_\ell(x)^2}, \frac{\omega_\ell(x,y)}{\psi_\ell(x)^3} \right)
\]
as in \cite[Exercise 3.7]{SilvermanEC} and consider the polynomial $g(x)=\phi_\ell(x)- x(P) \psi_\ell(x)^2 \in \mathbb{Q}[x]$. Suppose that this polynomial has an irreducible factor $g_1(x) \in \Q[x]$ of degree strictly less than $\frac{\ell^2}{2}$ (equivalently, for $\ell>2$, that $g(x)$ is reducible), and let $L$ be the field generated over $\Q$ by a root $x_1$ of $g_1(x)$. Over an at most quadratic extension $L'$ of $L$, the elliptic curve $E$ admits a point $Q$ with $x$-coordinate equal to $x_1$. It follows that $[\ell]Q = \left( \frac{\phi_\ell(x_1)}{\psi_\ell(x_1)^2} , y([\ell]Q) \right) = \left( x(P) , y([\ell]Q) \right) = \pm P$, because the only two points on $E$ with $x$-coordinate equal to $x(P)$ are $\pm P$.
In particular, at least one $\ell$-division point of $P$ (namely $\pm Q$) is defined over $L'$, which has degree strictly less than $\ell^2$ over $\Q$. Since all $\ell$-division points of $P$ are obtained from $\pm Q$ by adding a $\ell$-torsion point, the field $\Q_{\ell,\ell}$ is the compositum of $L'$ and $\Q(E[\ell])$, hence $[\Q_{\ell,\ell} : \Q(E[\ell])] \leq [L' : \Q] < \ell^2$. It follows that in this case the prime $\ell$ divides the ratio \eqref{eqn:kummerFailure} for $M=N=\ell$. 

We have considered several pairs $(E,P)$ taken from the LMFDB \cite{lmfdb}, and have computed (for well-chosen primes $\ell$) the factorisation of the polynomial $g(x)$ above. 
For each prime $\ell$ appearing as a factor of the constants of Theorem \ref{theorem:mainTheoremKummer}, we have thus been able to find examples of pairs $(E, P)$ for which $\ell$ divides the index \eqref{eqn:kummerFailure} in the case $M=N=\ell$, and this both for CM and non-CM curves (for $\ell=2$ we proceeded differently and explicitly computed the field generated by the $2$-division points of $P$; this easily yields examples). In particular, this shows that the prime factors of the constants of Theorem \ref{theorem:mainTheoremKummer} are all necessary.

We would like to point out that for most primes $\ell$ we have found several examples of the behaviour described above (for $\ell=163$ we have only been able to test two curves, and only one of them yielded an example). It is hard to make conjectures based on the limited evidence we have collected, but it seems plausible that $\ell$ divides the Kummer degree \eqref{eqn:kummerFailure} (with $M=N=\ell$) for a positive proportion of rank-1 curves $E/\Q$ whose mod-$\ell$ Galois representation lands in a Borel (when $P$ is taken to be a generator of the free part of $E(\Q)$). In Tables \ref{table:ExamplesnonCM} and \ref{table:ExamplesCM} we give one explicit example for every relevant prime, both for non-CM and CM curves, specifying the curve $E/\Q$ together with its LMFDB label and the point $P \in E(\Q)$.

\renewcommand{\arraystretch}{1.4}

\begin{table}
\begin{tabular}{c|c|c|c}
$\ell$ & $E$ & LMFDB Label & $P$ 
\\ \hline
$2$ & $y^2 + xy + y = x^3 - x^2 - 41x + 96$ & 
 \href{https://www.lmfdb.org/EllipticCurve/Q/117/a/3}{117.a3} & $(2,-6)$ \\
$3$ & $y^2 + y = x^3 + x^2 - 7x + 5$ & 
 \href{https://www.lmfdb.org/EllipticCurve/Q/91/b/2}{91.b2} & $(-1, 3)$ \\
$5$ & $y^2 = x^3 - x^2 - x - 1$ &
 \href{https://www.lmfdb.org/EllipticCurve/Q/704/c/3}{704.c3} & $(2, 1)$ \\
$7$ & $y^2 + xy = x^3 - x^2 - 389x - 2859$ &
 \href{https://www.lmfdb.org/EllipticCurve/Q/338/c/1}{338.c1} & $(26, 51)$ \\
$11$ & $y^2 + xy + y = x^3 - x^2 - 32693x - 2267130$ & 
 \href{https://www.lmfdb.org/EllipticCurve/Q/1089/c/1}{1089.c1} & $(212, 438)$ 
\\
$13$ & $y^2 + y = x^3 - 8211x - 286610$ &
 \href{https://www.lmfdb.org/EllipticCurve/Q/441/a/1}{441.a1} & $(235, 3280)$ 
\\
$17$ & $y^2 + xy + y = x^3 - x^2 - 27365x - 1735513$ &
 \href{https://www.lmfdb.org/EllipticCurve/Q/130050/gu/2}{130050.gu2} & $(\frac{4047}{4}, \frac{249623}{8})$ 
\\
$37$ & $y^2 + xy + y = x^3 + x^2 - 208083x - 36621194$ &
 \href{https://www.lmfdb.org/EllipticCurve/Q/1225/b/1}{1225.b1} & $(1190, 36857)$ 
\end{tabular}

\medskip

\caption{Primes $\ell$ dividing the relative Kummer degree \eqref{eqn:kummerFailure}, non-CM curves.\label{table:ExamplesnonCM}}
\end{table}

\begin{table}
\begin{tabular}{c|c|c|c}
$\ell$ & $E$ & LMFDB Label & $P$ 
\\ \hline
$2$ & $y^2 = x^3 - 36x$ &
 \href{https://www.lmfdb.org/EllipticCurve/Q/576/c/3}{576.c3} & $(-2,-8)$ \\
$3$ & $y^2 + y = x^3 - 34$ &
 \href{https://www.lmfdb.org/EllipticCurve/Q/225/c/1}{225.c1} & $(6, 13)$ \\
$7$ & $y^2 = x^3 - 1715x - 33614$ &
 \href{https://www.lmfdb.org/EllipticCurve/Q/784/f/2}{784.f2} & $(57, 232)$
\\
$11$ & $y^2 + y = x^3 - x^2 - 887x - 10143$ &
 \href{https://www.lmfdb.org/EllipticCurve/Q/121/b/1}{121.b1} & (81, 665) 
\\
$19$ & $y^2 + y = x^3 - 13718x - 619025$ &
 \href{https://www.lmfdb.org/EllipticCurve/Q/361/a/1}{361.a1} & $(2527, 126891)$ 
\\
$43$ & $y^2 + y = x^3 - 1590140x - 771794326$ &
 \href{https://www.lmfdb.org/EllipticCurve/Q/1849/b/1}{1849.b1} & $P_{43}$ 
\\
$67$ & $y^2 + y = x^3 - 33083930x - 73244287055$ &
 \href{https://www.lmfdb.org/EllipticCurve/Q/4489/b/1}{4489.b1} & $P_{67}$ 
\\
163 & $y^2 + y = x^3 - 57772164980x - 5344733777551611$ &
 \href{https://www.lmfdb.org/EllipticCurve/Q/26569/a/1}{26569.a1} & $P_{163}$
\end{tabular}

\medskip

\caption{Primes $\ell$ dividing the relative Kummer degree \eqref{eqn:kummerFailure}, CM curves.\label{table:ExamplesCM}}
\end{table}

The points $P_{43}$ and $P_{67}$ are given by $P_{43} = \displaystyle \left(\frac{66276734}{29929}, -\frac{419567566482}{5177717} \right)$ and 
\[
P_{67}=\displaystyle \left(\frac{49970077554856210455913}{1635061583290810756}, \frac{10956085084392718114395997318977993}{2090745506172424414999081096} \right)
\]
respectively. The point $P_{163}$ is the unique generator of $E(\Q) \cong \Z$ with positive $y$ coordinate; it has canonical height approximately equal to $373.48$, so its coordinates are too large to be displayed here, but they can be found at \cite[\href{https://www.lmfdb.org/EllipticCurve/Q/26569/a/1}{Elliptic Curve 26569.a1}]{lmfdb}.

We have also considered the divisibility of \eqref{eqn:kummerFailure} by higher powers of $\ell$. Experiments analogous to the above are computationally intensive, so we only studied the very small primes $2$ and $3$. An example where the index \eqref{eqn:kummerFailure} is divisible by $16$ was found by Rouse and Cerchia \cite{MR4198067}: letting $E : y^2 = x^3 - 343x + 2401$ and $P = (0, -49)$, there is a point $P_4 \in E(\mathbb{Q}(E[8]))$ such that $4P_4=P$. This implies that $2^4$ divides \eqref{eqn:kummerFailure} for $N=4, M=8$. We found several other examples in which \eqref{eqn:kummerFailure} is divisible by $2^4$ for suitable values of $M,N$, but no example involving higher powers of $2$. This might in part be due to the fact that -- for computational reasons -- we have only been able to extend our search to $M=8, N \mid M$.

\begin{remark}
J. Rouse recently informed us that he constructed an example where
\eqref{eqn:kummerFailure} is divisible by $2^6$ when $M$ and $N$ are
sufficiently large powers of $2$.
\end{remark}

For $\ell=3$ we consider $E: y^2 + y = x^3 - 6924x + 221760$ and $P=(2354/49, -176/343)$, which is a generator of $E(\Q)/E(\Q)_{\tors}$. Write $g(x)$ for the polynomial whose roots are the $x$-coordinates of the $9$-division points of $P$: one may check that $g(x) \in \mathbb{Q}[x]$ has an irreducible factor $g_1(x)$ of degree $9$. Further denote by $\psi_9(x)$ the $9$-th division polynomial of $E$, whose roots are the $x$-coordinates of the points in $E[9]$. We have also computed that the Galois groups of $\psi_9(x), g_1(x)$ and $\psi_9(x)g_1(x)$ over $\mathbb{Q}$ have order $462$, $27$ and $3 \cdot 462$ respectively. This proves that the Galois group of $g_1(x)$ over $\mathbb{Q}(E[9])$ has order 3, hence in particular that $g_1(x)$ becomes reducible over $\mathbb{Q}(E[9])$. A $9$-division point of $P$ is then defined over an extension of $\mathbb{Q}(E[9])$ of degree at most (and in fact exactly) 3. As before, all other 9-division points are defined over the same field, hence the relative Kummer degree \eqref{eqn:kummerFailure} is divisible by $3^3$ for $M=N=9$. We have found other examples where $3^3$ divides \eqref{eqn:kummerFailure}, but none involving a factor $3^4$; as with $\ell=2$, it is entirely possible that this is only due to the limits of our search range.

%% file: 7-3adic.tex

\section{Scalars in pro-$p$ subgroups of $\GL_2(\Z_p)$}\label{sect:pAdicAppendix}

In this appendix we prove an abstract group-theoretic result, used in Section \ref{subsect:3adicScalars} to study the subgroup of scalar matrices in the image of the $3$-adic representation attached to a non-CM elliptic curve over $\mathbb{Q}$. In the statement and proof of Proposition \ref{prop:padicscalars} we will employ the notation $H_{p^n}$ for the reduction modulo $p^n$ of a closed subgroup $H$ of $\GL_2(\Z_p)$ (cf.~Section \ref{subsect:GroupTheoreticCriteria}).

\begin{proposition}\label{prop:padicscalars}
Let $p$ be an odd prime, $H$ be a closed pro-$p$ subgroup of $\GL_2(\Z_p)$, and $k$ be a positive integer. Suppose that the following hold:
\begin{enumerate}
\item $H_p$ has order $p$,
\item $H_{p^k}$ is not contained in the subgroup of upper- or lower-triangular matrices;
\item $\det(H) = 1+p\Z_p$;
\item $H$ is normalised by $C := \matr100{-1}$.
\end{enumerate}
Then $H$ contains all scalars congruent to 1 modulo $p^k$.
\end{proposition}

\begin{remark}
From a group-theoretic point of view this result is optimal, at least in the case $p=3, k=3$ that we are interested in. The subgroup $H$ of $\GL_2(\Z_3)$ given by the inverse image of the subgroup of $\GL_2(\Z/3^3\Z)$ generated by the matrices
\[
\matr{10}00{16}, \quad \matr{10}{9}{23}{10}
\]
satisfies all the properties (1)-(4) in the statement, and
\[
H \cap \Z_3^\times = \{\lambda \in \Z_3^\times : \lambda \equiv 1 \pmod{3^3}\}.
\]
\end{remark}
\begin{remark}
We also note that the methods of \cite{MR1241950} and \cite[§4]{MR3437765} are not easily applicable here, since there is no reason to expect a group $H$ as in the statement of Proposition \ref{prop:padicscalars} to be open in $\GL_2(\Z_p)$. This implies that the $\mathbb{Z}_p$-integral Lie algebra $L$ attached to $H$ by \cite{MR1241950} could be quite small, with $L/[L,L]$ infinite, which makes it hard to extract useful information from the main theorem of \cite{MR1241950}.
\end{remark}

The proof of the proposition is by induction: we will show that, for every $n \geq k$, the group $H_{p^n}$ contains all scalars congruent to 1 modulo $p^k$. Since $H$ is closed this gives the desired conclusion.

\begin{remark}\label{rmk:MatrixD}
The group $H_p$ is cyclic, generated by any element $g$ of order $p$. The condition that $H$ be stable under conjugation by $C$ implies easily that $g$ is either upper- or lower-unitriangular (that is, triangular with diagonal coefficients equal to 1). This shows in particular that for every $h = \matr abcd \in H$ we have $a \equiv d \equiv 1 \pmod p$, so that the diagonal entries of $h-\Id$ are divisible by $p$. Any $h \in H$ may therefore be written as $h = \lambda \Id + D + A$, where $\lambda=\frac{1}{2} \operatorname{tr}(h) \equiv 1 \pmod{p}$, $D$ is diagonal, $\operatorname{tr}(D)=0$, $D \equiv 0 \pmod p$, and $A$ is anti-diagonal. This decomposition will play an important role in the proof.
\end{remark}

The following lemma will be key in our approach.
\begin{lemma}\label{lemma:padicKeyLemma}
Let $p$ be an odd prime, let $H_{p^n}$ be a $p$-subgroup of $\GL_2(\Z/p^n\Z)$ stable under conjugation by $C:=\matr100{-1}$, and let $M$ be an element of $H_{p^n}$. Consider the sequence of elements of $H_{p^n}$ defined by $M_0=M$ and $M_{i+1} = M_i \cdot C M_i C^{-1}$. Then:
\begin{enumerate}
\item for every $i \geq 0$, the elements $\det M_i$ and $\det M$ generate the same subgroup of $(\Z/p^n \Z)^\times$;
\item write each $M_i$ as $\lambda_i \Id + D_i + A_i$, where $D_i$ is diagonal and has trace $0$ and $A_i$ is anti-diagonal. Then there exists a scalar $\mu_i \in (\Z/p^n\Z)^\times$ such that $D_i=\mu_i D_0$;
\item the matrix $M_i$ is diagonal for all $i \geq n$.
\end{enumerate}
\end{lemma}
\begin{proof}
For the first statement we have $\det(M_{i+1})=\det(M_i)^2$ and the map $x\mapsto x^2$ is an automorphism of the abelian $p$-group $\det(H)$.
Write now $M_i=\lambda_i \Id + D_i + A_i$ as in the statement. It follows from Remark \ref{rmk:MatrixD} that $D_i \equiv 0 \pmod p$. One computes $CM_iC^{-1}=\lambda_i \Id + D_i - A_i$ and therefore
\[
\begin{aligned}
M_{i+1} & = \left( \lambda_i \Id + D_i + A_i \right) \left( \lambda_i \Id + D_i - A_i \right) \\
& = \lambda_i^2 + D_i^2 + 2\lambda_i D_i - A_i^2 + [A_i,D_i].
\end{aligned}
\]
Notice that $D_i^2$ is a multiple of the identity (since the two diagonal elements of $D_i$ are opposite to each other, hence have the same square), and so is $A_i^2$, while $[A_i, D_i]$ is anti-diagonal. Hence
\[
\begin{cases}
D_{i+1} = 2 \lambda_i D_i \\
A_{i+1} = [A_i, D_i],
\end{cases}
\]
which immediately implies the statement about $D_i$ since $(2\lambda_i,p)=1$. Moreover, since $v_p(D_i) \geq 1$ we have $v_p(A_{i+1}) \geq v_p(A_i)+1$: in particular, for $i \geq n$ we have $v_p(A_i) \geq n$, hence for such $i$ the matrix $A_i$ is $0$ and $M_i$ is diagonal.
\end{proof}
We notice in particular the following immediate consequence of the previous lemma:
\begin{corollary}\label{cor:DiagonalSurjectiveDeterminant}
Let $H_{p^n}$ be a $p$-subgroup of $\GL_2(\Z/p^n\Z)$ stable under conjugation by $C$, and let $\mathcal{D}_n$ be the subgroup of diagonal matrices in $H_{p^n}$. Then $\det(H_{p^n})=\det(\mathcal{D}_n)$.
\end{corollary}
\begin{proof}
The group $\det(H_{p^n})$ is contained in $(\Z/p^n\Z)^\times$, hence is cyclic. Let $M \in H_{p^n}$ be a matrix whose determinant generates $\det(H_{p^n})$: by the previous lemma, we can find a diagonal matrix whose determinant generates the same subgroup as $\det(M)$.
\end{proof}

Before proving Proposition \ref{prop:padicscalars} we need one further definition:
\begin{definition}\label{def:LieAlgebra}
For $n \geq 1$ we let $L_n$ be the image of the map
\[
\begin{array}{ccc}
\ker(H_{p^{n+1}} \to H_{p^n}) & \to & \Mat_{2 \times 2}(\F_p) \\
g & \mapsto & \frac{g-\Id}{p^n}.
\end{array}
\]
\end{definition}
The formulas
\[
(\Id+p^n M_1)(\Id + p^n M_2) \equiv \Id + p^n(M_1+M_2) \pmod{p^{n+1}}
\]
and $(\Id+p^n M)^p \equiv \Id + p^{n+1} M \pmod{p^{n+2}}$, valid for all $n \geq 1$, show that the set $L_n$ is an additive subgroup of $\Mat_{2 \times 2}(\F_p)$, and that moreover $L_n \subseteq L_{n+1}$ for all $n \geq 1$.

We further observe that since $C$ normalises $H$ the subspace $L_n$ of $\Mat_{2 \times 2}(\F_p)$ is stable under conjugation by $C$. Since $p$ is odd, the conjugation action of $C$ on $\Mat_{2 \times 2}(\F_p)$ decomposes it as the direct sum of the subspaces of diagonal and anti-diagonal matrices. We then have a corresponding decomposition $L_n = \mathfrak{d}_n \oplus \mathfrak{a}_n$, where $\mathfrak{d}_n$ (respectively $\mathfrak{a}_n$) is the subspace of diagonal (resp.~anti-diagonal) matrices in $L_n$.
We are now ready to begin the proof proper.

\begin{proof}[Proof of Proposition \ref{prop:padicscalars}]
We show by induction that $H_{p^n}$ contains all scalar matrices congruent to $1$ modulo $p^k$. Notice that the claim is trivial for $n \leq k$, so we only need to take care of the inductive step. For each positive integer $n$ we denote by $\mathcal{D}_{n}$ the subgroup of diagonal matrices in $H_{p^n}$ and by $\Lambda_n$ the subgroup $\{ \lambda \in (\Z/p^n\Z)^\times : \lambda \equiv 1 \pmod p \}$ of $(\Z/p^n\Z)^\times$. By Corollary \ref{cor:DiagonalSurjectiveDeterminant} and the hypothesis $\det(H) = 1+p\Z_p$ (hence $\det(H_{p^n}) = \Lambda_n$) we have $\# \mathcal{D}_n \geq \#\Lambda_n = p^{n-1}$ for all $n \geq 1$.
The kernel of the reduction map $\mathcal{D}_{n+1} \to \mathcal{D}_n$ is isomorphic to $\mathfrak{d}_n$ by construction. Notice that $\#\mathfrak{d}_n \in \{1,p,p^2\}$.

If $\#\mathfrak{d}_n = p^2$, the map $\mathcal{D}_{n+1} \to \mathcal{D}_{n}$ is $p^2$-to-$1$, which implies that, for every element in $\mathcal{D}_n$, \textit{all} its $p^2$ diagonal lifts to $\GL_2(\Z/p^{n+1}\Z)$ are in $\mathcal{D}_{n+1}$. In particular, since $(1+p^k)\Id \bmod \, p^n$ is an element of $\mathcal{D}_n$ by the inductive hypothesis and $(1+p^k)\Id \bmod \, p^{n+1}$ is one such possible lift, we obtain immediately that $(1+p^k)\Id$ is in $H_{p^{n+1}}$, and the induction step is complete (notice that the cyclic subgroup generated by $(1+p^k)\Id$ contains all scalars congruent to $1$ modulo $p^k$).

Suppose on the other hand that $\# \mathfrak{d}_n \mid p$. Then, using the fact that $\#\mathfrak{d}_i \mid \#\mathfrak{d}_{i+1}$, we obtain immediately
\[
\# \mathcal{D}_{n+1} = \#\mathcal{D}_1 \cdot \# \mathfrak{d}_1 \cdots \# \mathfrak{d}_n \mid p^n,
\]
which combined with our previous observation $\#\mathcal{D}_{n+1} \geq p^n$ implies $\#\mathcal{D}_{n+1}=p^n$. In particular,
\[
\det : \mathcal{D}_{n+1} \to \Lambda_{n+1}
\]
is a surjective group homomorphism between groups of the same order, hence is an isomorphism. This also implies that the only diagonal matrix in $H_{p^{n+1}}$ with determinant 1 is the identity.

Let now $d : \Lambda_{n+1} \to \mathcal{D}_{n+1}$ be the isomorphism given by the inverse of the determinant, which we write as
\[
d(x) = \begin{pmatrix}
\alpha(x) & 0 \\
0 & \beta(x)
\end{pmatrix}
\]
for suitable group homomorphisms $\alpha(x), \beta(x) : \Lambda_{n+1} \to \Lambda_{n+1}$. As $\Lambda_{n+1}$ is a cyclic group, we have $\alpha(x)=x^{a}$ and $\beta(x)=x^b$ for suitable integers $a, b$. Since $d(x)$ is inverse to the determinant, we have $x = \det(d(x))= \alpha(x)\beta(x) = x^{a+b}$, so that in particular $a+b$ is relatively prime to $p$. This implies that at least one between $a$ and $b$ is prime to $p$.

We now show that the intersection $S_{n+1} := H_{p^{n+1}} \cap \SL_2(\Z/p^{n+1}\Z)$ consists of matrices of the form $\lambda \operatorname{Id} + A$, where $\lambda \in \Z/p^{n+1}\Z$ is a scalar and $A$ is antidiagonal.
To see this, let $M \in S_{n+1}$, and write it as $M=\lambda \Id + D + A$, with $D$ diagonal of trace 0 and $A$ antidiagonal. Lemma \ref{lemma:padicKeyLemma} yields a diagonal matrix $M' = \lambda' \Id + D'$ in $S_{n+1}$ (in particular, $\det(M')=1$) with $D' = \mu D$ for some scalar $\mu$ prime to $p$. Since the only diagonal matrix with determinant 1 in $H_{p^{n+1}}$ is the identity, we get $\lambda' = 1$ and $D'=0$. As $\mu$ is invertible, this implies $D=0$ as desired.

On the other hand, $S_{n+1}$ -- being the kernel of the determinant -- is normal in $H_{p^{n+1}}$, hence in particular is stable under conjugation by the diagonal matrices $d(x)$. Let $M = \lambda \Id + A$ be any element of $S_{n+1}$ and let $x \in \Lambda_{n+1}$. Then $S_{n+1}$ also contains $d(x) \cdot M \cdot d(x)^{-1}$ and their product $M \cdot d(x) \cdot M \cdot d(x)^{-1}$, that is,
\begin{equation}\label{eq:ConjugationActionOnDerivedSubgroup}
(\lambda \Id +A)(\lambda \Id  + d(x) \cdot A \cdot d(x)^{-1}).
\end{equation}
Like all elements of $S_{n+1}$, this matrix has the form $\lambda' \Id + A'$ for some scalar $\lambda'$ and some anti-diagonal matrix $A'$. The diagonal part of \eqref{eq:ConjugationActionOnDerivedSubgroup} is $\lambda^2 + A \cdot d(x) \cdot A \cdot d(x)^{-1}$, so $A \cdot d(x) \cdot A \cdot d(x)^{-1}$ is a multiple of the identity modulo $p^{n+1}$. Writing $A=\begin{pmatrix}
0 & y \\ z & 0 
\end{pmatrix}$, the condition becomes
\begin{equation} \label{eq:AlmostScalarsp}
yz \left( \frac{\alpha(x)}{\beta(x)} - \frac{\beta(x)}{\alpha(x)} \right) \equiv 0 \pmod{p^{n+1}}.
\end{equation}
We will show below that there exists $M \in S_{n+1}$, $M= \lambda \Id + \matr 0yz0$, with $v_p(yz) \leq k-1$. Assuming for now that we have such an $M$, in Equation \eqref{eq:AlmostScalarsp} we may assume $v_p(yz) \leq k-1$, hence we obtain $\left( \frac{\alpha(x)}{\beta(x)} \right)^2 \equiv 1 \pmod{p^{n+2-k}}$. Recalling that $\alpha(x)=x^a, \beta(x)=x^b$, this rewrites as $x^a \equiv x^b \pmod{p^{n+2-k}}$ (notice that $x \mapsto x^2$ is an automorphism of $\Lambda_{n+1}$). Raising to the $p^{k-1}$-th power we get $x^{p^{k-1} a} \equiv x^{p^{k-1} b} \pmod{p^{n+1}}$, hence
\[
x^{p^{k-1} a} \operatorname{Id} = x^{p^{k-1} b} \operatorname{Id} = d \left( x^{p^{k-1}} \right) \in H_{p^{n+1}}
\]
for every $x \in \Lambda_{n+1}$. 
Recall now that at least one between $a$ and $b$ is prime to $p$, say $(a,p)=1$: then $x \mapsto x^a$ is an automorphism of $\Lambda_{n+1}$, so it follows that all the $p^{k-1}$-th powers of the scalars $\equiv 1 \pmod{p}$ are in $H_{p^{n+1}}$. The induction step is now complete, because all scalars congruent to $1$ modulo $p^k$ are $p^{k-1}$-th powers in $\Lambda_{n+1}$.

It only remains to show that we can find an element $M \in S_{n+1}$ such that, writing $M= \lambda \Id + \matr 0yz0$, we have $v_p(yz) \leq k-1$. We first prove that it is enough to find $N=\begin{pmatrix}
n_{11} & n_{12} \\
n_{21} & n_{22}
\end{pmatrix} \in H_{p^{n+1}}$ with $v_p(n_{12}n_{21}) \leq k-1$. Indeed, given such an $N$, we know from above that there is a diagonal matrix $Q=\begin{pmatrix}
q_{11} & 0 \\ 0 & q_{22}
\end{pmatrix} \in H_{p^{n+1}}$ with $\det(Q)=\det(N)^{-1}$. Notice that $q_{11}, q_{22}$ are invertible. Then $NQ=\begin{pmatrix}
q_{11} n_{11} & q_{22} n_{12} \\
q_{11} n_{21} & q_{22} n_{22}
\end{pmatrix}$ belongs to $S_{n+1}$, so it is automatically of the form $\lambda \Id+A$, and its anti-diagonal part satisfies $v_p(q_{22}n_{12} \, q_{11}n_{21}) = v_p(n_{12}n_{21}) \leq k-1$ as desired. Thus it suffices to find $N \in H_{p^{n+1}}$, of arbitrary determinant, with $v_p(n_{12}n_{21}) \leq k-1$.

By Remark \ref{rmk:MatrixD}, there exists $g \in H$ that reduces modulo $p$ to $\matr 1101$ or $\matr 1011$: for simplicity of exposition, we only discuss the former case, the latter being completely analogous. Consider the image $\begin{pmatrix} g_{11} & g_{12} \\ g_{21} & g_{22} \end{pmatrix}$ of $g$ in $H_{p^k}$: since $v_p(g_{12})=0$, if $v_p(g_{21}) \leq k-1$ we are done by taking $N=g \bmod {p^{n+1}}$. Otherwise, let $h \in H$ be an element whose image $\begin{pmatrix} h_{11} & h_{12} \\ h_{21} & h_{22} \end{pmatrix}$ in $H_{p^k}$ satisfies $v_p(h_{21}) \leq k-1$: such an element exists, for otherwise $H_{p^k}$ would be contained in the subgroup of upper-triangular matrices.
If $v_p(h_{12})=0$ we are done by taking $N=h \bmod{p^{n+1}}$, while if $v_p(h_{12})>0$ it is easy to check that we can take $N=hg \bmod{p^{n+1}}$.
\end{proof}
\begin{remark}
Part of the proof is inspired by the structure theorem for reductive groups. Indeed, in the course of the argument we prove that the diagonal torus of $H_{p^{n+1}}$ is isomorphic to $\Lambda_{n+1}$, which is the pro-$p$ subgroup of $\mathbb{G}_m(\Z/p^{n+1}\Z)$, that $S_{n+1} = H_{p^{n+1}} \cap \SL_2(\Z/p^{n+1}\Z)$ (morally, the derived subgroup) intersects the diagonal torus trivially, and finally that the conjugation action of the torus on the ``semisimple part'' $S_{n+1}$ is (essentially) trivial, so that the diagonal torus (essentially) consists of scalar matrices. This is reminiscent of the decomposition $G = Z(G).G'$ that holds for reductive groups, and indeed hypothesis (2) of the proposition may be seen as a discrete analogue of the statement ``$H$ is reductive".
\end{remark}